\RequirePackage{fix-cm}
\documentclass{svjour3}                     
\smartqed  
\usepackage{graphicx}
\usepackage[table]{xcolor}
\newcolumntype{a}{>{\columncolor{red}}c}
\makeatletter

\newcommand*{\@rowstyle}{}

\newcommand*{\rowstyle}[1]{
  \gdef\@rowstyle{#1}%
  \@rowstyle\ignorespaces%
}

\newcolumntype{=}{
  >{\gdef\@rowstyle{}}%
}

\newcolumntype{+}{
  >{\@rowstyle}%
}

\makeatother

 \usepackage{mathptmx}      
\usepackage{latexsym}
\usepackage{mathtools,amssymb}
\usepackage{float}
\usepackage{multirow}
\usepackage{graphicx}
\usepackage{subcaption}
\usepackage{natbib}
\begin{document}

\title{Strong convergence of a linearization method for semi-linear elliptic equations with variable scaled production \thanks{The first author was funded by US Army Research Laboratory and US Army Research Office grant W911NF-19-1-0044. The work of the first author was also 
supported by a postdoctoral fellowship of the Research Foundation-Flanders (FWO) in Belgium.}
}


\author{Vo Anh Khoa   \and
      Ekeoma Rowland Ijioma  \and
      Nguyen Nhu Ngoc
}


\institute{Vo Anh Khoa \at
              Department of Mathematics and Statistics, University of North Carolina at Charlotte, Charlotte, North Carolina 28223, USA      \\
			Faculty of Sciences, Hasselt University, Campus Diepenbeek, Agoralaan
			Building D, BE3590 Diepenbeek, Belgium             
           \\
           \email{vakhoa.hcmus@gmail.com, anhkhoa.vo@uncc.edu}
           \and
           Ekeoma Rowland Ijioma \at 
           Meiji Institute for Advanced Study of Mathematical Sciences, 4-21-1 Nakano, Nakano-ku, Tokyo, Japan \\
              \email{e.r.ijioma@gmail.com}
              \and
              Nguyen Nhu Ngoc \at
              Dipartimento di Matematica, Politecnico di Milano, Milano 20133, Italy\\
              \email{nhungoc.nguyen@polimi.it}   
}

\date{Received: date / Accepted: date}

\maketitle

\begin{abstract}
This work is devoted to the development and analysis of a linearization algorithm for microscopic elliptic equations, with scaled degenerate production, posed in a perforated medium and constrained by the homogeneous Neumann-Dirichlet boundary conditions. This technique plays two roles: to guarantee the unique weak solvability of the microscopic problem and to provide a fine approximation in the macroscopic setting. The scheme systematically relies on the choice of a stabilization parameter in such a way as to guarantee the strong convergence in $H^1$ norm for both the microscopic and macroscopic problems. In the standard variational setting, we prove the $H^1$-type contraction at the micro-scale based on the energy method. Meanwhile, we adopt the classical homogenization result in line with corrector estimate to show the convergence of the scheme at the macro-scale. In the numerical section, we use the standard finite element method to assess the efficiency and convergence of our proposed algorithm.
\end{abstract}

\section{Introduction}
\label{intro}
\subsection{Microscopic problem}
Let $ \Omega^{\varepsilon}$ be a Lipschitz perforated domain contained in a polygonal bounded domain $ \Omega \subset \mathbb{R}^d $ $(d=2, 3)$. In this sense, $ \Omega^{\varepsilon}$ possesses a
uniformly periodic microstructure defined by a length scale $\varepsilon$. This  $ \varepsilon $ is a small parameter $ (0<\varepsilon \ll 1) $ since the size of the pores are usually much smaller than the characteristic length of the reservoir. We are herein concerned with the asymptotic behavior in a stationary case of the function $u_{\varepsilon}: \Omega^{\varepsilon} \to \mathbb{R}$ that describes the spread of concentration of solutes dissolved in a saturated porous tissue shaped by the perforated domain $ \Omega^{\varepsilon}$ with a cubic periodicity
cell $Y = [0,1]^d$. The molecular diffusion coefficient $ \textbf{A}:Y\to \mathbb{R}^{d\times d}$ is assumed to vary in the cell $ Y $, while we consider in this scenario the presence of a volume reaction $\mathcal{R}: \mathbb{R}\to \mathbb{R} $ subject to an internal source $ f : \Omega \to \mathbb{R}$. We also take into account the no-flux boundary condition at the internal boundaries, denoted by $ \Gamma^{\varepsilon} $, whilst giving the homogeneous Dirichlet boundary condition at the exterior boundary, denoted by $\Gamma^{\text{ext}}$. Essentially, this context can be understood by the following elliptic problem:
\begin{equation}\label{Eq(P)}
(P_{\varepsilon}):  \begin{cases}
L_{\varepsilon}u_{\varepsilon}+ \varepsilon^{\alpha}\mathcal{R}(u_{\varepsilon}) = f(x) \; \qquad \text{in}\;\Omega^{\varepsilon}, \alpha \ge 0, \\
- \textbf{A}(x/\varepsilon) \nabla u_{\varepsilon} \cdot n = 0 \;\; \quad\qquad \;\; \text{across}\; \Gamma^{\varepsilon},\\
u_{\varepsilon}=0 \qquad \qquad\qquad\qquad\quad\;\; \text{across} \; \Gamma^{\text{ext}}, \\
\end{cases}
\end{equation}
where $ L_{\varepsilon} $ is a symmetric operator given by
\begin{align}
L_{\varepsilon} u=\nabla\cdot \left(-\textbf{A}\left(\frac{x}{\varepsilon}\right)\nabla u\right) = \sum_{i,j=1}^{d} \frac{\partial}{\partial x_i} \left(-a_{ij}\left(\frac{x}{\varepsilon}\right) \frac{\partial u}{\partial x_j}\right).
\end{align}

\subsection{Background and motivation} 

In this paper, we follow up on our earlier works \cite{Khoa2016,Khoa2017,khoa2018pore} that focus on the asymptotic analysis of semi-linear elliptic problems posed in perforated domains. This well-understood elliptic problem was applied in the studies of the heat transfer in composite materials and of the pressure and phase velocities in porous media flow. Even though those applications are involved in the microscopic system \eqref{Eq(P)}, our wish in this paper is to prepare and develop the playground of the model for diffusion, aggregation and surface deposition of a concentration in a porous medium; cf. e.g. \cite{Krehel2015,Timofte2013}. The presence of non-negative scalings stems from our mathematical concerns about the \emph{non-trivial} asymptotic behaviors of $u_{\varepsilon}$ when $\varepsilon$ tends to 0 and their corresponding rates. In fact, our most recent result in \cite{khoa2018pore} has shown that when $\alpha< 0$ the macroscopic solution is identically zero after the homogenization process ($\varepsilon\to 0$).

The study of the variable scaled nonlinearities soon appeared in the works \cite{Cabarrubias2012,C.Conca2004}, where they considered the homogenization of elliptic problems with a scaled Robin boundary condition. As a concrete motivation of our proposed model \eqref{Eq(P)}, the authors in \cite{C.Conca2004} delved into the chemical reactive flows through the exterior of a domain with distributed reactive obstacles, where the fractional (Langmuir kinetics) and polynomial (Freundlich kinetics) surface reactions were taken into account. Mathematically, such surface reactions posed on $\Gamma^{\varepsilon}$ read as $- \textbf{A}(x/\varepsilon) \nabla u_{\varepsilon} \cdot n = \varepsilon^{\beta}\mathcal{S}(u_{\varepsilon})$ for $\beta\ge 1$, just like what has been studied in \cite{khoa2018pore}. Note that even though we only consider in \eqref{Eq(P)} the zero Neumann case of the internal boundary $\Gamma^{\varepsilon}$, it is completely similar to adapt our analysis below to the surface reaction (using the same assumption as that of the volume one). Hence, our mathematical analysis in this work is pertinent and applicable.

Arguments obtained from studies of the variable scalings can be helpful in the qualitative analysis of eigen-elements for elliptic boundary value problems with rapidly oscillating coefficients in a perforated cube (cf. e.g. \cite{Douanla2012,Gryshchuk2013}), while it also concerns low-cost control problems in \cite{Muthukumar2009}. On top of that, they can be further adapted to complex scenarios that include, for example, the non-stationary Stokes--Nernst--Planck--Poisson system in the dynamics of colloids (cf. e.g. \cite{Ray2012} and references cited therein).

Cf. \cite{khoa2018pore}, the microscopic solution to the problem $(P_{\varepsilon})$ converges to a macroscopic function that solves a certain non-trivial homogenized problem since the scaling variable $\alpha\ge 0$ is eventually the main factor that determines the presence of the reaction term $\mathcal{R}$ at the macro-scale. In principle, the main purpose of homogenization is to unveil the macroscopic shape of $u_{\varepsilon}$ as it is less time-consuming than any approximation of $u_{\varepsilon}$ itself. Indeed, as is known in the homogenization community, solving $(P_{\varepsilon})$ directly is computationally expensive since the space discretization is inversely proportional to the scale parameter $\varepsilon$. Before the homogenization, one essentially needs the existence and uniqueness of the microscopic solution. Therefore, it ended up with our pursuit to design a linearization scheme that plays these two roles: simultaneously obtaining weak sovability of the microscopic system (for all $\varepsilon>0$) and the linearized homogenized system.

Cf. \cite{Khoa2016} as our initiation, a linearization scheme was briefly designed to prove the weak solvability of $(P_{\varepsilon})$ as $\alpha = 0$. However, this result was only guaranteed when the diffusion must be very larger than the Lipschitz rate of reactions. Our next evolution in this area went to the work \cite{khoa2018pore}, where, for the first time, we addressed a linearization scheme for the weak solvability of a semi-linear microscopic system with real variable scalings. Since in those works, our interests were in the asymptotic analysis of the microscopic solution, we did not get through the computational standpoint of the homogenized system. Thereupon, this is the moment to delineate a reliable approximation of $u_{\varepsilon}$ at the macro-scale, which is our second purpose of the aforementioned wish.

\subsection{Goals}
In this work, we show that the linearization scheme we design is essential in proving the well-posedness of $(P_{\varepsilon})$, but also in deriving the approximate macroscopic solution with certain error estimates. Basically, our theoretical analysis will proceed in accordance with the following diagram:
\begin{align}\label{d1}
\begin{array}{ccccccc}
\left(P_{\varepsilon}\right) & \longrightarrow & \left(P_{\varepsilon}^{k}\right) & \longrightarrow & \left(P_{0}^{k}\right) \\
\downarrow &  & \downarrow &  & \downarrow & \\
u_{\varepsilon} & \longrightarrow & u_{\varepsilon}^{k} & \longrightarrow & u_{0}^{k} 
\end{array}
\end{align}
Here, $(P_{\varepsilon}^{k})$ denotes the approximate problem of $(P_{\varepsilon})$ by linearization, whilst its macroscopic equation is structured in $(P_{0}^{k})$. The notion behind this approach is to linearize nonlinearities in the model using a suitable choice of the stabilization parameter. As the nonlinearity is supposed to be degenerate at a single point, we are aided by a regularization approach  during the linearization process. In this way, we arrive at a regularized form of the nonlinearity, where we can figure out the error estimate between $u_{\varepsilon}$ and $u_{\varepsilon}^{k}$ in $L^2$-norm. In our proof, the stabilization is $\varepsilon$-dependent only when the scaling factor $\alpha$ is positive. Meanwhile, $\varepsilon$ does not contribute to the convergence of the linearization scheme for any $\alpha \ge 0$. As $\varepsilon\to 0$ in the homogenization process, the stabilization constant becomes $\varepsilon$-independent for any $\alpha \ge 0$. Henceforth, several stability estimates for the macroscopic solution are easily obtained.

As another vantage of our proposed scheme, we point out that if the solvability of $(P_{\varepsilon})$ and the corresponding macroscopic equation are already known, one can also use the scheme directly to get the approximation $u_{0}^{k}$. In this case, we mean
\begin{align}\label{d2}
\begin{array}{ccccccc}
\left(P_{\varepsilon}\right) & \longrightarrow & \left(P_{0}\right) & \longrightarrow & \left(P_{0}^{k}\right) \\
\downarrow &  & \downarrow &  & \downarrow & \\
u_{\varepsilon} & \longrightarrow & u_{0} & \longrightarrow & u_{0}^{k} 
\end{array}
\end{align}

Furthermore, in the diagram \eqref{d2} we can prove that the rate of convergence is $k$-independent (cf. \cite{Slodika2001}), compared to the result we obtain in  Corollary \ref{cor:error2} for the diagram \eqref{d1}. It is worth mentioning that as we aim to show the unique weak solvability of the microscopic problem by a linearization technique, we follow diagram \eqref{d1}.

\subsection{Outline of the paper}
The paper is organized as follows. Section \ref{Sec:2} is dedicated to introducing notations and necessary assumptions on the input of the problem. In Section \ref{Sec:3}, we propose an iterations-based variational scheme to linearize the microscopic problem $\left(P_{\varepsilon}\right)$. Accordingly, we obtain the well-posedness of $\left(P_{\varepsilon}\right)$ as well as the rate of convergence by the linearization we choose. Settings of the homogenization are involved in Section \ref{Sec:4}, where we also state the structures of the cell problems and the limit equation at every step of linearization. Additionally, several types of stability analysis of the scheme at the macro-scale are justified. Section \ref{Sec:numerical} is devoted to the numerical test  of the scheme and we close this paper with some concluding remarks in Section \ref{Sec:conclusion}.

\section{Preliminaries}\label{Sec:2}
In the sequel, all the constants $C$ are  independent of the scaling parameter $\varepsilon$,  but their
precise values may differ from line to line and may change even within a single chain
of estimates. We use either the superscript or subscript $\varepsilon$ to indicate its dependence.  Depending on the situation, we denote by $\left|\cdot\right|$ the absolute value of a function, the finite-dimensional Euclidean norm of a vector or the
volume of a domain.

The perforated domain $\Omega^{\varepsilon}$ is supposed to be bounded, connected and possesses a $C^{0,1}$-boundary. It is thought to approximate a porous medium and its precise description can be found in \cite{Khoa2015}. For brevity, we herein skip the mathematical descriptions of the perforated domain of interest. Instead, we only provide Figure \ref{fig:1} an admissible geometry of our medium and the corresponding microstructure when looking for a graphically schematic representation of the scaling procedure within a natural soil.  In addition, the reader can be referred to \cite{Hornung1991,Fatima2014,Cancedda2016} and those collectively mentioned in Section \ref{intro} for some concrete results concerning such domains.

\begin{figure}[htbp]
	\centering
	\includegraphics[scale = 0.45]{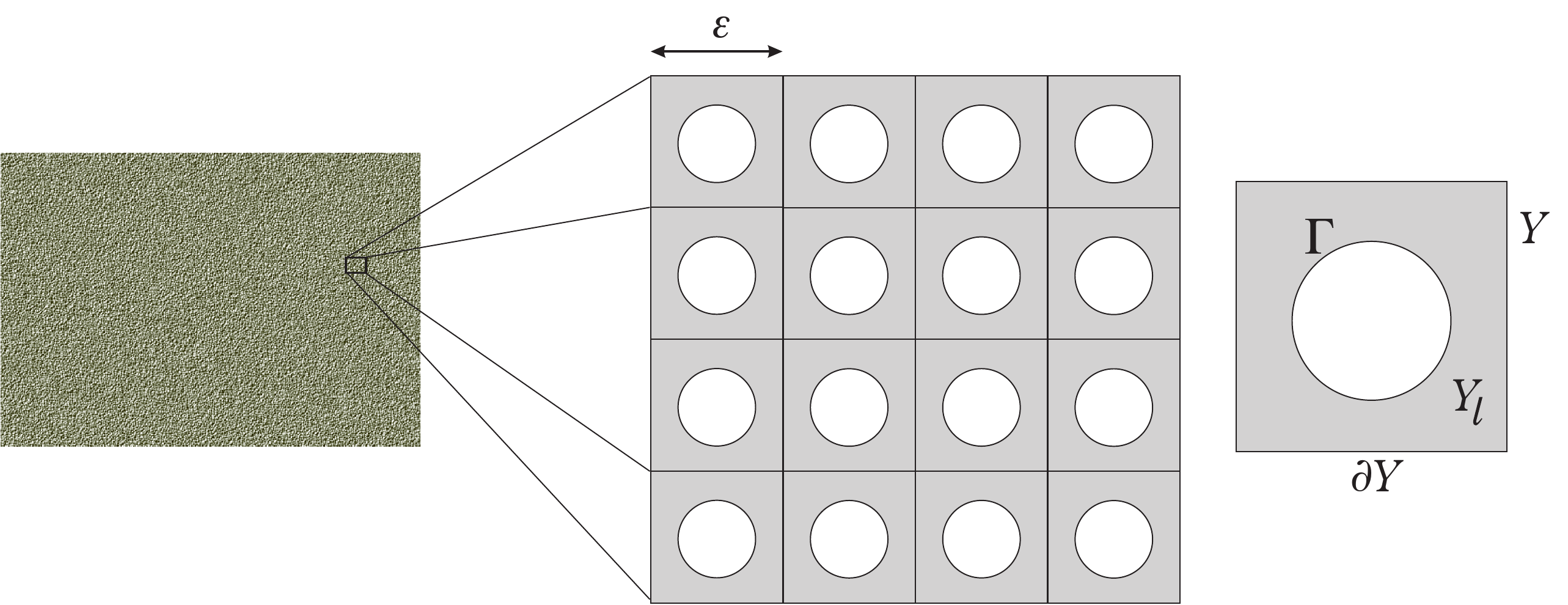}
	\caption{A schematic representation of the scaling procedure.}
	\label{fig:1}
\end{figure}

\begin{definition}{(Degenerate class)}\label{def:degenerate}
	A real-valued function $F$ is said to be degenerate at a point $x_{0}\in \mathbb{R}$ if we can find $\delta_{1}>0$ such that $0\le F'(x)\le \delta_1$ a.e. in $\mathbb{R}$ and the following conditions hold true:
	\begin{itemize}
		\item If $\lim\limits_{x\to s^{+}}F'\left(x\right)\lim\limits_{x\to s^{-}}F'\left(x\right)=0$
		for some $s\in\mathbb{R}$, then $s=x_{0}$.
		\item There exist $\delta_{0},r_{x_{0}}>0$ such that
		\begin{align}
		0<\delta_{0}\le F'\left(x\right)\le\delta_{1}\quad\text{for a.e. }x\in\mathbb{R}\backslash\mathcal{B}\left(x_{0},r_{x_{0}}\right),
		\end{align}
		where $\mathcal{B}\left(x_{0},\delta_{x_0}\right)$ denotes a ball centered
		at $x_{0}$ with a radius $r_{x_{0}}$.
		\item For $x\in \overline{\mathcal{B}\left(x_{0},r_{x_0}\right)}$, $F'$ is non-decreasing.
	\end{itemize}
\end{definition}

Next, we introduce the space
\begin{align}
V^{\varepsilon}:=\left\{ u\in H^{1}\left(\Omega^{\varepsilon}\right):u=0\;\text{on }\Gamma^{\text{ext}}\right\}, 
\end{align}
equipped with the norm
\begin{align}
\left\Vert u\right\Vert _{V^{\varepsilon}}=\left(\int_{\Omega^{\varepsilon}}\left|\nabla u\right|^{2}dx\right)^{1/2}.
\end{align}

Cf. \cite[Lemma 2.1]{Cioranescu1999}, one can show the uniform-in-$\varepsilon$ equivalence between this norm and the usual $H^1$-norm by the Poincar\'e-type inequality.

\begin{lemma}\label{lem:poincare}
There exists a constant $C_p>0$ independent of $\varepsilon$ such that
	\begin{align}
	\left\Vert u\right\Vert^{2} _{L^{2}\left(\Omega^{\varepsilon}\right)}\le C_p\left\Vert \nabla u\right\Vert^{2} _{L^{2}\left(\Omega^{\varepsilon}\right)}\quad\text{for any }u\in V^{\varepsilon}.
	\end{align}
\end{lemma}

Moreover, we denote by $H^1_{0}(\Omega)$ the Hilbert space of weakly differentiable functions $u:\Omega \to \mathbb{R}$ that vanishes on the boundary in the sense of trace. We also use below the space $H^1_{\#}(Y_{l})$ to indicate functions in $H^1(Y_{l})$ that is $Y$-periodic and has zero mean value.

This is now the moment to state our working assumptions on data involved in the microscopic problem. Those include

$\left(\text{A}_{1}\right)$ The diffusion $\mathbf{A}$ is essentially
bounded, $Y$-periodic, symmetric and globally Lipschitz. It satisfies
the uniform ellipticity condition in the sense that we can find $\varepsilon$-independent
constants $\underline{\gamma},\overline{\gamma}>0$ such that
\begin{align}
\underline{\gamma}\left|\xi\right|^{2}\le \sum_{i,j=1}^{d} a_{ij}\left(\frac{x}{\varepsilon}\right)\xi_{i}\xi_{j}\le\overline{\gamma}\left|\xi\right|^{2}\quad\text{for all }\xi\in\mathbb{R}^{d}\;\text{and }1\le i,j\le d.
\end{align}

$\left(\text{A}_{2}\right)$ The reaction term $\mathcal{R}:\mathbb{R}\to\mathbb{R}$
is degenerate.

$\left(\text{A}_{3}\right)$ The internal source $f$ belongs to $L^2(\Omega)$.

\begin{remark}
The consideration of a degenerate class of the reaction/production term $\mathcal{R}$ starts from our mathematical wish to generalize the previous results in \cite{Khoa2015,khoa2018pore}, where we exploited linearization schemes under strong assumptions, e.g the global Lipschitz reaction accounting for the Monod equation. In \cite{Khoa2015}, we even needed the fact that the fraction of reaction and diffusion bounds has to be sufficiently small. Our generalization in this paper is based upon the molecular interaction of piecewise types. For example, they can be of the power-growth law $F(x) = \left|x\right|^{p},p>1$ in a neighborhood of the degenerate point $x_0$ (cf. Definition \ref{def:degenerate}), and follow any global Lipschitz rate outside this neighborhood. A certain example of this typical reaction is introduced in \eqref{eq:3.7}.
\end{remark}

\section{Settings of the iterative variational algorithm}\label{Sec:3}

There are several linearization methods investigated in the past and each with its modifications and improvements to serve certain classes of nonlinear partial differential equations. To give a very cursory glance, one may concern the Newton method, cf. \cite{Bergamaschi1999}, whose convergence requires the initial guess to be close to the true solution, albeit its quadratic convergence. The J\"ager--Ka\v cur scheme is also renowned for its outstanding performance in the approximations of one-dimensional parabolic problems with a linear convergence; see \cite{Jger1995}. In this paper, our method is conventionally in line with the so-called $L$-scheme extensively studied in many distinctive types of parabolic equations; cf. \cite{Mitra2018} and references cited therein for a short background concerning this typical scheme. In this sense, the so-called stabilization term is added to stabilize the entire linearized equation in the standard variational formulation. Thereby, a linear convergence is obtained under a suitable choice of the stabilization constant.

\begin{definition}\label{Defweaksol}
	For each $ \varepsilon > 0 $, a function $ u_{\varepsilon} \in H^1(\Omega^{\varepsilon})$ is said to be a weak solution to $ (P_{\varepsilon}) $ if it satisfies
	\begin{align}\label{eq:3.1}
	a(u_\varepsilon, \varphi)  + \varepsilon^{\alpha}	\langle \mathcal{R}(u_\varepsilon),\varphi \rangle_{L^2(\Omega^{\varepsilon})} = \langle f,\varphi \rangle_{L^2(\Omega^{\varepsilon})},
	\end{align}
	for all $\varphi \in V^{\varepsilon}$, where $ a $: $ H^1(\Omega^{\varepsilon}) \times H^1(\Omega^{\varepsilon}) \to \mathbb{R} $ is a bilinear mapping given by
	\begin{align}
	a(u,\varphi) := \int_{\Omega^{\varepsilon}} \textbf{A}(x/\varepsilon) \nabla u\cdot \nabla \varphi dx.
	\end{align}
\end{definition}

At this stage, we take into account the degeneracy of the reaction term $\mathcal{R}$, especially in the ball where its derivative is zero and non-decreasing (see again in Definition \ref{def:degenerate}). It is worth citing here the Jackson type estimates in the approximation theory of monotone functions by monotone polynomials. In principle, cf. \cite{Leviatan1988}, for every monotone non-decreasing function $\bar{f}\in C^{k}[-1,1]$, there are non-decreasing polynomials $p_n$, whose degree does not exceed $n$, such that
\begin{align}\label{hehe}
\sup_{x\in\left[-1,1\right]}\left|\bar{f}-p_{n}\right|\le Cn^{-k}\omega\left(\bar{f}^{(k)},n^{-1}\right),
\end{align}
where $C$ here is independent of $\bar{f}$ and $n$, and $\omega$ indicates the modulus of continuity of $\bar{f}^{(k)}$. Cf. \cite{Kimchi1976,Roulier1976} for $\bar{f}\in C^{2k}[-1,1]$ with certain conditions, if $\bar{f}$ possesses strict lower and upper bounds of the available derivatives, then for sufficiently large $n$, the best polynomial approximation to $\bar{f}$ also satisfies the same property. Besides, the corresponding derivatives of the best approximation approach the derivatives of $\bar{f}$, respectively. Cf. \cite{Leviatan2000} for a survey of recent developments of the polynomial approximation, the approximation process can also preserve the monotonicity of $\bar{f}'$.

The aforementioned references enable us to assume the existence of a regularization scheme for the degenerate $\mathcal{R}$. In this regard, a function $F_{\gamma}$ for $\gamma>0$ being as a regularization parameter is said to be a regularization of a degenerate function $F$ if one has
\begin{itemize}
	\item $0< \gamma \delta_0 \le  F_{\gamma}'(x)\le \delta_{1}$ for any $x\in \mathbb{R}$,
	\item $\left|F-F_{\gamma}\right|\le C\gamma^{\sigma}$ for  $\sigma>0$ and for any $x\in\mathbb{R}$.
\end{itemize}

Technically, the smallness of the lower bound of $F_{\gamma}'$ is taken to regularize the degeneracy of $F$ at $x_0$ as it is zero at this degenerate point. Thus, one may tacitly look for $F_{\gamma}$ such that $F_{\gamma}'(x_{0}) = \gamma\delta_{0}$ and should attempt to preserve the  ``shape" of $F$ through the regularization process. Note that this approach does not mean that the regularization scheme must be a linear mapping; in general, the process can still be nonlinear.

After regularization of the reaction term, the linearized problem in the variational setting is given as follows.

\begin{definition}\label{Defapproxweaksol}
	For each $\varepsilon>0$, a linear approximation of $u_{\varepsilon}$ that solves \eqref{eq:3.1} is defined as a sequence $\left\{ u_{\varepsilon}^{k}\right\} _{k\in\mathbb{N}^{*}}$ satisfying
	\begin{align}\label{prob: P^k_{varepsilon}}
	 a(u_\varepsilon^k, \varphi) + M	\langle u_\varepsilon^k,\varphi \rangle_{L^2(\Omega^{\varepsilon})}   = \langle f,\varphi \rangle_{L^2(\Omega^{\varepsilon})}  
	+  M\langle u_\varepsilon^{k-1},\varphi \rangle_{L^2(\Omega^{\varepsilon})}
	 - \varepsilon^{\alpha} \langle \mathcal{R}_{\gamma_{k}}(u^{k-1}_\varepsilon),\varphi \rangle_{L^2(\Omega^{\varepsilon})},
	\end{align}
	where the stabilization constant $M\ge 0$ and the regularization parameter $\gamma_{k}>0$ are selected later. The initial guess is taken as $u_{\varepsilon}^{0}=0$.
\end{definition}

In the following theorem, we prove that this sequence is well-defined and exists uniquely in $H^1(\Omega^{\varepsilon})$.

\begin{theorem}\label{thm:3.1-1}
	Assume $\left(\text{A}_{1}\right)$--$\left(\text{A}_{3}\right)$ hold. By choosing $M = \eta + \varepsilon^{\alpha}\delta_{1}$ for $\eta >0$, the approximate problem \eqref{prob: P^k_{varepsilon}} admits at most a weak solution in $H^1(\Omega^{\varepsilon})$.
\end{theorem}
\begin{proof}
	It suffices to consider the first-loop problem of  (\ref{prob: P^k_{varepsilon}}) (i.e. $ k=1 $), which reads as
	\begin{align}
	a\left(u_{\varepsilon}^{1},\varphi\right)+M\left\langle u_{\varepsilon}^{1},\varphi\right\rangle _{L^{2}\left(\Omega^{\varepsilon}\right)}=\left\langle \tilde{f},\varphi\right\rangle _{L^{2}\left(\Omega^{\varepsilon}\right)},
	\end{align}
	where $\tilde{f}:= f - \varepsilon^{\alpha}\mathcal{R}_{\gamma_{k}}(0)$. Hereby, we can introduce the bilinear form $ \mathcal{B}: H^1\left(\Omega^{\varepsilon}\right)\times H^1\left(\Omega^{\varepsilon}\right) \to \mathbb{R}$ given by $\mathcal{B}(u,\varphi):= a(u,\varphi) + M\left\langle u,\varphi\right\rangle _{L^{2}\left(\Omega^{\varepsilon}\right)}$.
	Hence, we complete the proof of the theorem by virtue of the natural $\varepsilon$-independent coerciveness and continuity of $\mathcal{B}$ in $H^1(\Omega^{\varepsilon})$.
\end{proof}



\begin{lemma}\label{lem:3.1}
	Let $\left\{ p_{k}\right\} _{k\in\mathbb{N}^{*}}$ and $\left\{ q_{k}\right\} _{k\in\mathbb{N}^{*}}$ be sequences of nonnegative real numbers that obey the following recursion
	\begin{align}
	p_{k}+q_{k}\le a_{k}+b_{k}q_{k-1},\quad k\ge 2,
	\end{align}
	where $a_{k}$ and $b_k$ are also nonnegative real numbers. Then, it holds
	\begin{align}
	p_{k}+q_{k}\le a_{k}+\sum_{j=2}^{k-1}a_{j}\prod_{i=j+1}^{k}b_{i}+q_{1}\prod_{i=2}^{k}b_{i}, \quad k\ge 3.
	\end{align}
\end{lemma}
\begin{proof}
	The proof is trivial and it can be found in \ref{app:1}.
\end{proof}

\begin{theorem}\label{thm:3.1}
	Let $\alpha \ge 0$. Assume $\left(\text{A}_{1}\right)$--$\left(\text{A}_{3}\right)$ hold and let $\gamma_k>0$ be non-increasing with $\sigma>0$. Then by choosing
	$M = \eta + \varepsilon^{\alpha}\delta_{1}$ for $\eta >0$,
	the sequence $\left\{ u_{\varepsilon}^{k}\right\} _{k\in\mathbb{N}^{*}}$ of the variational problem \eqref{prob: P^k_{varepsilon}} possesses the following property
	\begin{align}\label{eq:3.4-1}
	&\frac{\underline{\gamma}}{\eta + \delta_1 + \underline{\gamma}C_p^{-1}}\left\Vert \nabla u_{\varepsilon}^{k} - \nabla u_{\varepsilon}^{k-1}\right\Vert _{L^{2}\left(\Omega^{\varepsilon}\right)}^{2}
	+\left\Vert u_{\varepsilon}^{k} - u_{\varepsilon}^{k-1}\right\Vert _{L^{2}\left(\Omega^{\varepsilon}\right)}^{2}\\ \nonumber
	& \le \frac{C\gamma_{k-1}^{2\sigma}}{\gamma_{k}} +C\sum_{j=2}^{k-1}\left(\frac{M}{\underline{\gamma}C_p^{-1}+M}\right)^{k-j}\frac{\gamma_{j-1}^{2\sigma}}{\gamma_{j}}
	+\left\Vert u_{\varepsilon}^{1}\right\Vert _{L^{2}\left(\Omega^{\varepsilon}\right)}^{2}\left(\frac{M}{\underline{\gamma}C_p^{-1}+M}\right)^{k-1}.
	\end{align}
\end{theorem}
\begin{proof}
	Consider $w_{\varepsilon}^{k}:=u_{\varepsilon}^{k}-u_{\varepsilon}^{k-1}$ for $k\ge 2$ as a difference function between the $k$th and $(k-1)$th steps of approximation. Then the difference equation is provided by
	\begin{align}\label{eq:3.3}
	a(w_\varepsilon^k, \varphi) + M	\langle w_\varepsilon^k,\varphi \rangle_{L^2(\Omega^{\varepsilon})} 	
	=   M\langle w_\varepsilon^{k-1},\varphi \rangle_{L^2(\Omega^{\varepsilon})}
	 - \varepsilon^{\alpha} \langle \mathcal{R}_{\gamma_{k}}(u^{k-1}_\varepsilon)-\mathcal{R}_{\gamma_{k-1}}(u^{k-2}_\varepsilon),\varphi \rangle_{L^2(\Omega^{\varepsilon})}.
	\end{align}
	
	Define  $g_{\gamma_k}\left(t\right):=\varepsilon^{\alpha}\mathcal{R}_{\gamma_k}\left(t\right)-Mt$ for $t\in\mathbb{R}$. By taking the test function $\varphi = w_{\varepsilon}^{k}$, \eqref{eq:3.3} becomes
	\begin{align*}
	 & \int_{\Omega^{\varepsilon}}\mathbf{A}\left(\frac{x}{\varepsilon}\right)\left|\nabla w_{\varepsilon}^{k}\right|^{2}dx+M\left\Vert w_{\varepsilon}^{k}\right\Vert _{L^{2}\left(\Omega^{\varepsilon}\right)}^{2}\\&
	 =\varepsilon^{\alpha}\left\langle \mathcal{R}_{\gamma_{k-1}}\left(u_{\varepsilon}^{k-2}\right)-\mathcal{R}_{\gamma_{k}}\left(u_{\varepsilon}^{k-2}\right),w_{\varepsilon}^{k}\right\rangle _{L^{2}\left(\Omega^{\varepsilon}\right)} +\left\langle g_{\gamma_{k}}\left(u_{\varepsilon}^{k-2}\right)-g_{\gamma_{k}}\left(u_{\varepsilon}^{k-1}\right),w_{\varepsilon}^{k}\right\rangle _{L^{2}\left(\Omega^{\varepsilon}\right)}.
	\end{align*}
	
	Observe that $\left|g'_{\gamma_{k}}\right|\le M-\varepsilon^{\alpha}\gamma_{k}\delta_{0}$ and in view of the fact that
	\begin{align}
	\left|\mathcal{R}_{\gamma_{k-1}}-\mathcal{R}_{\gamma_{k}}\right|\le C\left(\gamma_{k-1}^{\sigma}+\gamma_{k}^{\sigma}\right),
	\end{align}
	resulting from the regularization factor we apply, we estimate that
	\begin{align*}
	 &\underline{\gamma}\left\Vert \nabla w_{\varepsilon}^{k}\right\Vert _{L^{2}\left(\Omega^{\varepsilon}\right)}^{2}+M\left\Vert w_{\varepsilon}^{k}\right\Vert _{L^{2}\left(\Omega^{\varepsilon}\right)}^{2}
	 \\ &\le \underbrace{C\varepsilon^{\alpha}\left(\gamma_{k-1}^{\sigma}+\gamma_{k}^{\sigma}\right)\left\Vert w_{\varepsilon}^{k}\right\Vert _{L^{2}\left(\Omega^{\varepsilon}\right)}}_{:=I_{1}}+
	 \underbrace{\left(M-\varepsilon^{\alpha}\gamma_{k}\delta_{0}\right)\left\Vert w_{\varepsilon}^{k-1}\right\Vert _{L^{2}\left(\Omega^{\varepsilon}\right)}\left\Vert w_{\varepsilon}^{k}\right\Vert _{L^{2}\left(\Omega^{\varepsilon}\right)}}_{:=I_{2}}.
	\end{align*}
	
	Upon the monotonicity of $\gamma_{k}$, we use the Young inequality to get
	\begin{align*}
	I_{1}  \le\frac{C\varepsilon^{\alpha}\gamma_{k-1}^{2\sigma}}{\delta_{0}\gamma_{k}}+\frac{\varepsilon^{\alpha}\delta_{0}\gamma_{k}}{2}\left\Vert w_{\varepsilon}^{k}\right\Vert _{L^{2}\left(\Omega^{\varepsilon}\right)}^{2}.
	\end{align*}
	
	Using again the Young inequality, we also obtain the upper bound of $I_2$ as follows:
	\begin{align*}
	I_{2}  \le\frac{M-\varepsilon^{\alpha}\gamma_{k}\delta_{0}}{2}\left(\left\Vert w_{\varepsilon}^{k-1}\right\Vert _{L^{2}\left(\Omega^{\varepsilon}\right)}^{2}+\left\Vert w_{\varepsilon}^{k}\right\Vert _{L^{2}\left(\Omega^{\varepsilon}\right)}^{2}\right).
	\end{align*}
	
	Thereby, after some rearrangements, we find that
	\begin{align}\label{eq:3.4} 
	2\underline{\gamma}\left\Vert \nabla w_{\varepsilon}^{k}\right\Vert _{L^{2}\left(\Omega^{\varepsilon}\right)}^{2}+M\left\Vert w_{\varepsilon}^{k}\right\Vert _{L^{2}\left(\Omega^{\varepsilon}\right)}^{2} \le\frac{C\gamma_{k-1}^{2\sigma}}{\delta_{0}\gamma_{k}}\varepsilon^{\alpha}
	+\left(M-\varepsilon^{\alpha}\gamma_{k}\delta_{0}\right)\left\Vert w_{\varepsilon}^{k-1}\right\Vert _{L^{2}\left(\Omega^{\varepsilon}\right)}^{2}.
	\end{align}
	
	At present, we take $M = \eta + \varepsilon^{\alpha}\delta_{1}$ for $\eta>0$ (independent of $\varepsilon$ and $k$) in \eqref{eq:3.4}. Then we apply the Poincar\'e inequality (cf. Lemma \ref{lem:poincare}) to arrive at
	\begin{align}
	\frac{\underline{\gamma}}{\eta + \delta_1 + \underline{\gamma}C_p^{-1}}\left\Vert \nabla w_{\varepsilon}^{k}\right\Vert _{L^{2}\left(\Omega^{\varepsilon}\right)}^{2}+
	\left\Vert w_{\varepsilon}^{k}\right\Vert _{L^{2}\left(\Omega^{\varepsilon}\right)}^{2}\le\frac{C\gamma_{k-1}^{2\sigma}}{\gamma_{k}}\varepsilon^{\alpha}
	+ \frac{M}{\underline{\gamma}C_p^{-1} + M}\left\Vert w_{\varepsilon}^{k-1}\right\Vert _{L^{2}\left(\Omega^{\varepsilon}\right)}^{2},
	\end{align}
	and it thus follows from Lemma \ref{lem:3.1} that
	\begin{align}\label{eq:3.5}
	 \frac{\underline{\gamma}}{\eta + \delta_1 + \underline{\gamma}C_p^{-1}}&\left\Vert \nabla w_{\varepsilon}^{k}\right\Vert _{L^{2}\left(\Omega^{\varepsilon}\right)}^{2}
	+\left\Vert w_{\varepsilon}^{k}\right\Vert _{L^{2}\left(\Omega^{\varepsilon}\right)}^{2}\\ \nonumber
	& \le a_{k}+\sum_{j=2}^{k-1}a_{j}\prod_{i=j+1}^{k}b_{i}
	+\left\Vert u_{\varepsilon}^{1}\right\Vert _{L^{2}\left(\Omega^{\varepsilon}\right)}^{2}\prod_{i=2}^{k}b_{i},\quad k\ge 3,
	\end{align}
	where we have denoted by
	\begin{align*}
	a_{k}  :=\frac{C\gamma_{k-1}^{2\sigma}}{\gamma_{k}},\;
	b_{k}  := \frac{M}{\underline{\gamma}C_p^{-1} + M}.
	\end{align*}
	
	Naturally, $b_k\in(0,1)$ and thus the product of $b_i$ approaches 0 when $k$ tends to infinity in the sense that
	\begin{align}
	\prod_{i=j+1}^{k}b_{i}=\left(\frac{M}{\underline{\gamma}C_p^{-1}+M}\right)^{k-j}\quad\text{for }j\ge1.
	\end{align}
	This completes the proof of the theorem.
\end{proof}

In \eqref{eq:3.4-1}, we observe that the stability of the scheme is essentially dependent on the partial sums of the series of $a_k$, particularly of the choice of the regularization parameter $\gamma_{k}$. Since in this paper we obtain a strong convergence along with a particular error estimate, harmonic series are not reliable in ensuring that the sequence $\left\{ u_{\varepsilon}^{k}\right\} _{k\in\mathbb{N}^{*}}$ is Cauchy in $H^1\left(\Omega^{\varepsilon}\right)$. This is hindered by the convergence-towards-zero of the series of $a_k$ when using the standard triangle inequality. In simpler terms, harmonic series (or even hyperharmonic series) are mostly either divergent or convergent to a non-zero constant. It is worth noting that the product of $b_k$ possesses an exponential-like decay as $k\to \infty$. Then the same behavior should be applied to the series of $a_k$ by looking for a geometric progression of $a_k$.



As a concrete example, we take into account the power-law reaction rate in a unit domain, which reads as
\begin{align}\label{eq:3.7}
\mathcal{R}\left(u\right)=\begin{cases}
u^{p} & \text{for }u\in\left[0,1\right], p>1, \\
u & \text{for }u<0,\\
u & \text{for }u>1.
\end{cases}
\end{align}
Therefore, in the interval $[0,1]$ we can choose a regularization of $\mathcal{R}$ as follows:
\begin{align}
\mathcal{R}_{\gamma_{k}}(u)=\max\left\{ u^{p},\frac{\delta_{0}}{2^{(p-1)^2 + (k+1)(p-1)}}u\right\},\; k\in \mathbb{N}^{*},
\end{align}
leading to the fact that $\frac{\delta_0}{2^{(p-1)^2 + (k+1)(p-1)}}\le \mathcal{R}_{\gamma_{k}}' \le 1$ and the following estimate
\begin{align}
\left|\mathcal{R}\left(u\right)-\mathcal{R}_{\gamma_{k}}\left(u\right)\right|\le p^{\frac{1}{1-p}}\left|1-p\right|\left(\frac{\delta_{0}}{2^{(p-1)^2 + (k+1)(p-1)}}\right)^{\frac{p}{p-1}}.
\end{align}
In this way, we indicate that $\gamma_{k}=\frac{1}{2^{(p-1)^2 + (k+1)(p-1)}}$ and $\sigma = 1+\frac{1}{p-1}>1$. It is worth noting that
\begin{align}
\frac{\gamma_{k-1}^{2\sigma}}{\gamma_{k}^{2}}=\frac{1}{2^{2(p-1)}}\frac{2^{2(k+1)(p-1)}}{2^{2k(p-1)}}4^{-k}=4^{-k},
\end{align}
and thereupon, we, in accordance with the estimate \eqref{eq:3.4-1}, provide the following theorem.

\begin{theorem}\label{thm:3.2}
	Under the assumptions of Theorem \ref{thm:3.1}, if we further choose $\gamma_{k}$ in such a way that
	\begin{align}
	\frac{\gamma_{k-1}^{\sigma}}{\gamma_{k}}\le C\omega^{k},\quad\text{for any } k\in\mathbb{N}^{*},\omega \in(0,1),
	\end{align}
	and suppose that
	\begin{align}\label{dk1}
	\overline{\omega}:=\omega + \sqrt{\frac{M}{\underline{\gamma}C_p^{-1}+M}} < 1.
	\end{align}
	Then, the iterative sequence $\left\{ u_{\varepsilon}^{k}\right\} _{k\in\mathbb{N}^{*}}$ is Cauchy in $H^1\left(\Omega^{\varepsilon}\right)$. Consequently, the microscopic problem $ (P_{\varepsilon}) $ admits a unique solution $u_{\varepsilon}$ in $H^1\left(\Omega^{\varepsilon}\right)$.
\end{theorem}
\begin{proof}
	According to Theorem \ref{thm:3.1}, it is straightforward to find an $\varepsilon$-independent upper bound for \eqref{eq:3.4-1} as follows:
	\begin{align*}
	\sqrt{\frac{\underline{\gamma}}{\eta+\delta_{1}+\underline{\gamma}C_{p}^{-1}}}\left\Vert \nabla u_{\varepsilon}^{k}-\nabla u_{\varepsilon}^{k-1}\right\Vert _{L^{2}\left(\Omega^{\varepsilon}\right)} +\left\Vert u_{\varepsilon}^{k}-u_{\varepsilon}^{k-1}\right\Vert _{L^{2}\left(\Omega^{\varepsilon}\right)}
	 \le C\left(\omega+\sqrt{\frac{M}{\underline{\gamma}C_{p}^{-1}+M}}\right)^{k-1},
	\end{align*}
	where we have essentially used the binomial identity.
	Thereby, for any $k,r\in \mathbb{N}^{*}$ we obtain
	\begin{align}\label{eq:3.6}
	\sqrt{\frac{\underline{\gamma}}{\eta + \delta_1 + \underline{\gamma}C_p^{-1}}}& \left\Vert \nabla\left(u_{\varepsilon}^{k+r}-u_{\varepsilon}^{k}\right)\right\Vert _{L^{2}\left(\Omega^{\varepsilon}\right)}+\left\Vert u_{\varepsilon}^{k+r}-u_{\varepsilon}^{k}\right\Vert _{L^{2}\left(\Omega^{\varepsilon}\right)}\\\nonumber
	& \le C\left(\overline{\omega}^{k+r-1}+\overline{\omega}^{k+r-2}+\ldots+\overline{\omega}^{k}\right)\le\frac{C\overline{\omega}^{k}\left(1-\overline{\omega}^{r}\right)}{1-\overline{\omega}},
	\end{align}
	using the standard triangle inequality. This way we prove that $\left\{ u_{\varepsilon}^{k}\right\} _{k\in\mathbb{N}^{*}}$ is Cauchy in $H^1\left(\Omega^{\varepsilon}\right)$. Consequently, there exists a unique function $u_{\varepsilon}\in H^1\left(\Omega^{\varepsilon}\right)$ such that $u_{\varepsilon}^{k}\to u_{\varepsilon}$ as $k\to \infty$. Furthermore, it is straightforward to get that
	\begin{align*} \varepsilon^{\alpha}\mathcal{R}_{\gamma_{k}}\left(u_{\varepsilon}^{k-1}\right)\to\varepsilon^{\alpha}\mathcal{R}\left(u_{\varepsilon}\right)\;\text{strongly in }L^{2}\left(\Omega^{\varepsilon}\right).
	\end{align*}
	Thus, it enables us to confirm the existence and uniqueness of $u_{\varepsilon}$ to the microscopic problem $ (P_{\varepsilon}) $ in $H^1\left(\Omega^{\varepsilon}\right)$. This completes the proof of the theorem.
\end{proof}

\begin{corollary}\label{cor:error1}
	Under the assumptions of Theorem \ref{thm:3.2}, the following rate of convergence holds
	\begin{align}
	\sqrt{\frac{\underline{\gamma}}{\eta + \delta_1 + \underline{\gamma}C_p^{-1}}}\left\Vert \nabla\left(u_{\varepsilon}-u_{\varepsilon}^{k}\right)\right\Vert _{L^{2}\left(\Omega^{\varepsilon}\right)}+\left\Vert u_{\varepsilon}-u_{\varepsilon}^{k}\right\Vert _{L^{2}\left(\Omega^{\varepsilon}\right)}\le\frac{C\overline{\omega}^{k}}{1-\overline{\omega}}.
	\end{align}
\end{corollary}
\begin{proof}
	The proof of this corollary is obvious by the aid of \eqref{eq:3.6} when taking $r\to \infty$.
\end{proof}

\begin{remark}\label{rem:3.1}
	It is worth mentioning here that if we already know that $u_{\varepsilon}\in H^1(\Omega^{\varepsilon})$ is a unique solution of the microscopic problem $(P_{\varepsilon}) $, then the partial sums of the series $\sum_{j=2}^{k-1}a_{j}\prod_{i=j+1}^{k}b_{i}$ (cf. \eqref{eq:3.5}) just needs to approach 0 as $k\to \infty$ with a certain convergence rate, by considering the energy-like estimate of the difference between the linearized problem and $(P_{\varepsilon}) $. Accordingly, the whole error estimate is controlled by such a rate of that partial sums. This exactly mimics the proof in \cite{Slodika2001} where a minimal polynomial rate $k^{-\omega}$ for $\omega\in (0,1)$, which basically leads to the harmonic progression, is sufficiently taken into account. 
	
	Additionally, it is straightforward to obtain the stability analysis of the scheme $\left\{ u_{\varepsilon}^{k}\right\} _{k\in\mathbb{N}^{*}}$ we construct above, which, in principle, provides concretely $\varepsilon$-independent $\emph{a priori}$ estimates in $H^1(\Omega^{\varepsilon})$. This somewhat enables us to get the existence of a weak solution to the problem $(P_{\varepsilon}) $ in $H^1(\Omega^{\varepsilon})$ by the standard compactness argument, if we are able to derive, at least, the weak convergence of the reaction term in $L^2(\Omega^{\varepsilon})$ after passing to the limit. However, the uniqueness result may not always be achievable by this strategy. In this work, we do not go beyond this matter and will leave it for the future works.
\end{remark}

\section{Settings of the homogenization}
\label{Sec:4}

In the previous section, the rigorous error estimate for the linearization scheme has been obtained in $H^1(\Omega^{\varepsilon})$; cf. Corollary \ref{cor:error1}. In this section, we only exploit the $L^2$ error estimate, although it is well-known from the corrector estimate for the homogenization limit that the microscopic solution of a linear elliptic equation approaches the macroscopic solution of the corresponding homogenized elliptic equation with a rate $\mathcal{O}\left(\varepsilon^{\frac{1}{2}}\right)$ in $H^1(\Omega^{\varepsilon})$; cf. \cite[Corollary 2.29]{Cioranescu1999}. From here on, it is very easy to adapt this result since our approximate problem defined in \eqref{Defapproxweaksol} is all linear at every step $k$.

Our goal here is to introduce the structures of the non-trivial homogenized problem $(P_{0})$ for $(P_{\varepsilon}) $ as well as its cell problems for the sake of computations in Section \ref{Sec:numerical}. When doing so, we remark that $u_{\varepsilon}$ satisfies an $\emph{a priori}$ estimate by means of $\left\Vert u_{\varepsilon}\right\Vert _{H^{1}\left(\Omega^{\varepsilon}\right)} \le C$ established in Theorem \ref{thm:3.1-1} and by taking into account the usual zero extension on $u_{\varepsilon}$ from $H^1(\Omega^{\varepsilon})$ to $H^1(\Omega)$. Accordingly, we only need to take the test function $\varphi = \psi_0 (x) + \varepsilon \psi_{1}\left(x,\frac{x}{\varepsilon}\right)$ for $\psi_0,\psi_{1} \in \mathcal{D}(\Omega;C^{\infty}_{\#}(Y))$ in \eqref{Defapproxweaksol}. Henceforward, the compactness result  allows us to extract subsequences from bounded sequences and to obtain the passage to the two-scale limit. For detailed results concerning the two-scale convergence method for the linear elliptic equation, we refer the interested reader to e.g. \cite{Henning2009,Ray2012} under the theoretical results of the two-scale convergence postulated in \cite{Allaire1992,Nguetseng1989}.

After plugging that typical test function we have
\begin{align*}
& \int_{\Omega^{\varepsilon}}\textbf{A}\left(\frac{x}{\epsilon}\right)\nabla u_{\varepsilon}^{k}\cdot\nabla\left(\psi_{0}\left(x\right)+\varepsilon\psi_{1}\left(x,\frac{x}{\varepsilon}\right)\right)dx+M\int_{\Omega^{\varepsilon}}u_{\varepsilon}^{k}\left(\psi_{0}\left(x\right)+\varepsilon\psi_{1}\left(x,\frac{x}{\varepsilon}\right)\right)dx\\
& =\int_{\Omega^{\varepsilon}}f\left(\psi_{0}\left(x\right)+\varepsilon\psi_{1}\left(x,\frac{x}{\varepsilon}\right)\right)dx+M\int_{\Omega^{\varepsilon}}u_{\varepsilon}^{k-1}\left(\psi_{0}\left(x\right)+\varepsilon\psi_{1}\left(x,\frac{x}{\varepsilon}\right)\right)dx\\
& -\varepsilon^{\alpha}\int_{\Omega^{\varepsilon}}\mathcal{R}_{\gamma_{k}}\left(u_{\varepsilon}^{k-1}\right)\left(\psi_{0}\left(x\right)+\varepsilon\psi_{1}\left(x,\frac{x}{\varepsilon}\right)\right)dx,
\end{align*}
then by passing to the limit $\varepsilon\to 0$, we are led to the following limit equation:
\begin{itemize}
	\item Case $\alpha > 0$:
	\begin{align*} & \int_{\Omega}\int_{Y_{l}}\mathbf{A}(y)\left(\nabla_{x}u_{0}^{k}\left(x\right)+\nabla_{y}u_{1}^{k}\left(x,y\right)\right)\cdot\left(\nabla_{x}\psi_{0}\left(x\right)+\nabla_{y}\psi_{1}\left(x,y\right)\right)dydx\\
	& +\eta\int_{\Omega}\int_{Y_{l}}u_{0}^{k}\left(x\right)\psi_{0}\left(x\right)dydx=\int_{\Omega}\int_{Y_{l}}f\left(x\right)\psi_{0}\left(x\right)dydx\\
	& +\eta\int_{\Omega}\int_{Y_{l}}u_{0}^{k-1}\left(x\right)\psi_{0}\left(x\right)dydx.
	\end{align*}
	\item Case $\alpha = 0$:
	\begin{align*}
	& \int_{\Omega}\int_{Y_{l}}\mathbf{A}(y)\left(\nabla_{x}u_{0}^{k}\left(x\right)+\nabla_{y}u_{1}^{k}\left(x,y\right)\right)\cdot\left(\nabla_{x}\psi_{0}\left(x\right)+\nabla_{y}\psi_{1}\left(x,y\right)\right)dydx\\
	& +\left(\eta+\delta_{1}\right)\int_{\Omega}\int_{Y_{l}}u_{0}^{k}\left(x\right)\psi_{0}\left(x\right)dydx=\int_{\Omega}\int_{Y_{l}}f\left(x\right)\psi_{0}\left(x\right)dydx\\
	& +\left(\eta+\delta_{1}\right)\int_{\Omega}\int_{Y_l}u_{0}^{k-1}\left(x\right)\psi_{0}\left(x\right)dydx+\int_{\Omega}\int_{Y_{l}}\mathcal{R}_{\gamma_{k}}\left(u_{0}^{k-1}\right)\psi_{0}\left(x\right)dydx.
	\end{align*}
\end{itemize}

These two cases are almost the same since we are at the linearization stage. Typically, we choose $\psi_0 = 0$ in both two cases to get
\begin{align}
\begin{cases}
\nabla_{y}\cdot\left(-\mathbf{A}(y)\left(\nabla_{x}u_{0}^{k}\left(x\right)+\nabla_{y}u_{1}^{k}\left(x,y\right)\right)\right)=0 & \text{in }\Omega\times Y_{l},\\
-\mathbf{A}(y)\left(\nabla_{x}u_{0}^{k}\left(x\right)+\nabla_{y}u_{1}^{k}\left(x,y\right)\right)\cdot\text{n}=0 & \text{on }\Omega\times\Gamma,\\
u_{1}^{k}\left(x,y\right)\;\text{is periodic in }y,
\end{cases}
\end{align}
where we have used the integration by parts with respect to $y$. In this way, we use the separation of variables to find
\begin{align}
u_{1}^{k}\left(x,y\right)=\sum_{i=1}^{d}\chi^{k}_{i}\left(y\right)\partial_{x_{i}} u_{0}^{k}\left(x\right).
\end{align}
Here, $\chi^{k}_{i}$ is called as the cell function that solves the following
cell problem:
\begin{align}\label{eq:cell}
\begin{cases}
\nabla_{y}\cdot\left(-\mathbf{A}\left(y\right)\left(\nabla_{y}\chi_{i}^{k}\left(y\right)+e_{i}\right)\right)=0 & \text{in }Y_{l},\\
-\mathbf{A}\left(y\right)\left(\nabla_{y}\chi_{i}^{k}\left(y\right)+e_{i}\right)\cdot\text{n}=0 & \text{on }\Gamma,\\
\chi_{i}^{k}\left(y\right)\;\text{is periodic}.
\end{cases}
\end{align}

It is worth noting that the cell problem for every step $k$ remains unchanged, so that our computations will be less expensive in the sense that we do not need to compute the vector-valued $\chi = \chi^{k}=\left(\chi_{i}^{k}\right)_{1\le i\le d}$ at every $k$. Furthermore, one can prove that such $\chi^{k}\in H^1_{\#}(Y_{l})$. Due to the non-convexity of $Y_{l}$, the regularity of the unique function $\chi$ stops at $H^{1+s}(Y_{l})$ for $s\in (-1/2,1/2)$, no matter how smooth the involved data are; see \cite{Savare1998} for detailed concerns.

Now, choosing $\psi_{1} = 0$ and then applying the integration by parts with respect to $x$, we obtain the equation for $u_{0}^{k}$ as follows:
\begin{align}\label{eq:4.2}
& \nabla\cdot\left(-A^{0}\nabla u_{0}^{k}\right)\\
& +\begin{cases}
\eta\left|Y_{l}\right|u_{0}^{k}=\left|Y_{l}\right|f+\eta\left|Y_{l}\right|u_{0}^{k-1} & \text{if }\alpha>0,\\ \nonumber
\left(\eta+\delta_{1}\right)\left|Y_{l}\right|u_{0}^{k}=\left|Y_{l}\right|f+\left(\eta+\delta_{1}\right)\left|Y_{l}\right|u_{0}^{k-1}-\left|Y_{l}\right|\mathcal{R}_{\gamma_{k}}\left(u_{0}^{k-1}\right) & \text{if }\alpha=0,
\end{cases}
\end{align}
posed in $\Omega$. Here, this equation is endowed with the Dirichlet boundary condition $u_0^k = 0$ at $\partial\Omega$ and $A^0$ is known as the homogenized coefficient given by
\begin{align}\label{eq:homogenized}
a^{0}_{ij}:=\int_{Y_{l}}A\left(y\right)\left(\partial_{y_i}\chi\left(y\right)+\delta_{ij}\right)dy,
\end{align}
where $\delta_{ij}$ stands for the constants of the identity matrix. Additionally, this coefficient satisfies the uniform ellipticity condition by virtue of the well-known Voigt--Reiss inequality; see e.g. \cite{Sviercoski2008}. Thus, proof of the well-posedness of the macroscopic problem in $H^1_0(\Omega)$ is standard.

\begin{remark}
It is worth noting that the linearization scheme of the microscopic system \eqref{Eq(P)} plays two roles in this paper. As the primary goal, it allows us to prove the well-posedness of the microscopic system for any $\varepsilon> 0$. The second purpose of this scheme is to obtain the corresponding approximate homogenized system \eqref{eq:4.2}. To this end, this finding enables us to compute the reliable homogenized solution. This is a new story of the linearization scheme developed in \cite{Slodika2001} for semi-linear problems at only the macro-scale, where the author proved its strong convergence by a priori knowledge of the well-posedness of the original problem. Moreover, this is a novel contribution of the recent result \cite{khoa2018pore}, where the authors only studied the global Lipschitz reaction and the asymptotic analysis of \eqref{Eq(P)}.
\end{remark}

\begin{remark}
	As we now work on the linearized elliptic problem at the macro-scale, we below follow the same strategy of Theorems \ref{thm:3.1}, \ref{thm:3.2} and those of \cite{Slodika2001} to guarantee the stability (in some suitable norms) of the scheme in the macroscopic scaling and with respect to the number of iteration $k$.
\end{remark}

\begin{theorem}\label{thm:4.2}
	Under the assumptions of Theorem \ref{thm:3.2} and suppose that
	\begin{align}\label{dk2}
	\omega + \sqrt{\frac{(\eta + \delta_1)\left|Y_{l}\right|}{\left|A^{0}\right|c_p^{-1}+(\eta+\delta_1)\left|Y_{l}\right|}} < 1,
	\end{align}
	where $c_p>0$ is the standard Poincar\'e\footnote{For any $u\in H_{0}^{1}\left(\Omega\right)$, it holds $\left\Vert u\right\Vert^{2} _{L^{2}\left(\Omega\right)}\le c_{p}\left\Vert \nabla u\right\Vert^{2} _{L^{2}\left(\Omega\right)}$.} constant. Then, the iterative sequence $\left\{u_{0}^{k}\right\}_{k\in\mathbb{N}^{*}}$ is Cauchy in $H^1_0(\Omega)$.
\end{theorem}
\begin{proof}
	It is trivial to prove that the functions $u_{0}^{k}$ are all in $H_0^1(\Omega)$. Observe that \eqref{eq:4.2} is structured by the cases of $\alpha$, we thus also divide the proof here into two parts. In the first part, we treat the following equation:
	\begin{align}
	\nabla\cdot\left(-A^{0}\nabla u_{0}^{k}\right)+\eta\left|Y_{l}\right|u_{0}^{k}=\left|Y_{l}\right|f+\eta\left|Y_{l}\right|u_{0}^{k-1}.
	\end{align}
	Recall that the problem under consideration is associated with the zero Dirichlet boundary condition.
	
	We now use the test function $\bar{\varphi}\in H^1_0(\Omega)$ to arrive at the following variational formulation:
	\begin{align}
	\bar{a}\left(u_{0}^{k},\bar{\varphi}\right)+\eta\left|Y_{l}\right|\left\langle u_{0}^{k},\bar{\varphi}\right\rangle _{L^{2}\left(\Omega\right)}=\left|Y_{l}\right|\left\langle f,\bar{\varphi}\right\rangle _{L^{2}\left(\Omega\right)}+\eta\left|Y_{l}\right|\left\langle u_{0}^{k-1},\bar{\varphi}\right\rangle _{L^{2}\left(\Omega\right)},
	\end{align}
	where $\bar{a}:H_0^1(\Omega)\times H_0^1(\Omega)\to \mathbb{R}$ is a bilinear mapping given by
	\begin{align}
	\bar{a}\left(u,\bar{\varphi}\right):=A^{0}\int_{\Omega}\nabla u\cdot\nabla\bar{\varphi}dx.
	\end{align}
	
	Then, it is straightforward to compute the difference equation by putting $v_{0}^{k} = u_0^k - u_0^{k-1}$. This function essentially satisfies the following equation:
	\begin{align}
	\bar{a}\left(v_{0}^{k},\bar{\varphi}\right)+\eta\left|Y_{l}\right|\left\langle v_{0}^{k},\bar{\varphi}\right\rangle _{L^{2}\left(\Omega\right)}=\eta\left|Y_{l}\right|\left\langle v_{0}^{k-1},\bar{\varphi}\right\rangle _{L^{2}\left(\Omega\right)}\; \text{for }\bar{\varphi}\in H^1_0(\Omega),
	\end{align}
	and then, by taking $\bar{\varphi}=v_0^k$ it leads to
	\begin{align}\label{eq:4.4}
	\left|A^{0}\right|\left\Vert \nabla v_{0}^{k}\right\Vert _{L^{2}\left(\Omega\right)}^{2}+\eta\left|Y_{l}\right|\left\Vert v_{0}^{k}\right\Vert _{L^{2}\left(\Omega\right)}^{2}  &=\eta\left|Y_{l}\right|\left\langle v_{0}^{k-1},v_{0}^{k}\right\rangle_{L^2(\Omega)} \nonumber\\
	&\le\frac{\eta\left|Y_{l}\right|}{2}\left(\left\Vert v_{0}^{k-1}\right\Vert _{L^{2}\left(\Omega\right)}^{2}+\left\Vert v_{0}^{k}\right\Vert _{L^{2}\left(\Omega\right)}^{2}\right).
	\end{align}
	
	Due to the standard Poincar\'e inequality, we have
	\begin{align}
	\frac{\left|A^{0}\right|}{\left|A^{0}\right|c_p^{-1}+\eta\left|Y_{l}\right|}\left\Vert \nabla v_{0}^{k}\right\Vert _{L^{2}\left(\Omega\right)}^{2}+\left\Vert v_{0}^{k}\right\Vert _{L^{2}\left(\Omega\right)}^{2}\le\frac{\eta\left|Y_{l}\right|}{\left|A^{0}\right|c_p^{-1}+\eta\left|Y_{l}\right|}\left\Vert v_{0}^{k-1}\right\Vert _{L^{2}\left(\Omega\right)}^{2},
	\end{align}
	and by using Lemma \ref{lem:3.1}, we estimate that
	\begin{align}
	\frac{\left|A^{0}\right|}{\left|A^{0}\right| c_p^{-1}+\eta\left|Y_{l}\right|}\left\Vert \nabla v_{0}^{k}\right\Vert _{L^{2}\left(\Omega\right)}^{2}+\left\Vert v_{0}^{k}\right\Vert _{L^{2}\left(\Omega\right)}^{2}\le\left(\frac{\eta\left|Y_{l}\right|}{\left|A^{0}\right|c_p^{-1}+\eta\left|Y_{l}\right|}\right)^{k-1}\left\Vert u_{0}^{1}\right\Vert _{L^{2}\left(\Omega\right)}^{2}.
	\end{align}
	Similar to proof of Theorem \ref{thm:3.2}, we can prove that $\left\{u_{0}^{k}\right\}_{k\in\mathbb{N}^{*}}$ is Cauchy in $H^1_0(\Omega)$.
	
	Next, we consider the following equation:
	\begin{align}
	\nabla\cdot\left(-A^{0}\nabla u_{0}^{k}\right)+\left(\eta+\delta_{1}\right)\left|Y_{l}\right|u_{0}^{k}=\left|Y_{l}\right|f+\left(\eta+\delta_{1}\right)\left|Y_{l}\right|u_{0}^{k-1}-\left|Y_{l}\right|\mathcal{R}_{\gamma_{k}}\left(u_{0}^{k-1}\right).
	\end{align}
	
	We proceed as above by taking into account the variational formulation of the difference equation. By so doing, \eqref{eq:4.4} becomes
	\begin{align*}
	&\left|A^{0}\right|\left\Vert \nabla v_{0}^{k}\right\Vert _{L^{2}\left(\Omega\right)}^{2}  +\left(\eta+\delta_{1}\right)\left|Y_{l}\right|\left\Vert v_{0}^{k}\right\Vert _{L^{2}\left(\Omega\right)}^{2}\\
	& =\left|Y_{l}\right|\left\langle \mathcal{R}_{\gamma_{k-1}}\left(u_{0}^{k-2}\right)-\mathcal{R}_{\gamma_{k}}\left(u_{0}^{k-2}\right),v_{0}^{k}\right\rangle _{L^{2}\left(\Omega\right)} +\left|Y_{l}\right|\left\langle \bar{g}_{\gamma_{k}}\left(u_{0}^{k-2}\right)-\bar{g}_{\gamma_{k}}\left(u_{0}^{k-1}\right),v_{0}^{k}\right\rangle _{L^{2}\left(\Omega\right)},
	\end{align*}
	where we have denoted by  $\bar{g}_{\gamma_k}(t)=\mathcal{R}_{\gamma_k}(t)-\left(\eta + \delta_1\right)t$.
	
	Since $\left|\bar{g}_{\gamma_k}'\right| \le \eta + \delta_1 -\gamma_k\delta_0$ and $\left|\mathcal{R}_{\gamma_{k-1}}-\mathcal{R}_{\gamma_{k}}\right|\le C\gamma_{k-1}^{\sigma}$, we then apply the Young inequality to obtain
	\begin{align*}
	\left|A^{0}\right|\left\Vert \nabla v_{0}^{k}\right\Vert _{L^{2}\left(\Omega\right)}^{2} & +\left(\eta+\delta_{1}\right)\left|Y_{l}\right|\left\Vert v_{0}^{k}\right\Vert _{L^{2}\left(\Omega\right)}^{2}\\
	& \le\frac{C\gamma_{k-1}^{2\sigma}}{\gamma_{k}}+\left(\frac{\eta+\delta_{1}}{2}\right)\left|Y_{l}\right|\left(\left\Vert v_{0}^{k-1}\right\Vert _{L^{2}\left(\Omega\right)}^{2}+\left\Vert v_{0}^{k}\right\Vert _{L^{2}\left(\Omega\right)}^{2}\right),
	\end{align*}
	and after some rearrangements and applying the standard Poincar\'e inequality, we arrive at
	\begin{align*}
	\frac{\left|A^{0}\right|}{\left(\eta+\delta_{1}\right)\left|Y_{l}\right|+\left|A^{0}\right|c_{p}^{-1}}\left\Vert \nabla v_{0}^{k}\right\Vert _{L^{2}\left(\Omega\right)}^{2} & +\left\Vert v_{0}^{k}\right\Vert _{L^{2}\left(\Omega\right)}^{2}\\
	& \le\frac{C\gamma_{k-1}^{2\sigma}}{\gamma_{k}}+\frac{\left(\eta+\delta_{1}\right)\left|Y_{l}\right|}{\left(\eta+\delta_{1}\right)\left|Y_{l}\right|+\left|A^{0}\right|c_{p}^{-1}}\left\Vert v_{0}^{k-1}\right\Vert _{L^{2}\left(\Omega\right)}^{2}.
	\end{align*}
	
	Henceforward, with the aid of Lemma \ref{lem:3.1} we have the same estimate as \eqref{eq:3.4-1} and by the choice of $\gamma_{k}$ in Theorem \ref{thm:3.2}, we consequently prove that $\left\{u_{0}^{k}\right\}_{k\in\mathbb{N}^{*}}$ is Cauchy in $H^1_0(\Omega)$. This completes the proof of the theorem.
\end{proof}

As argued in Remark \ref{rem:3.1}, the geometric progression is required to prove the existence and uniqueness of the microscopic problem $\left(P_{\varepsilon}\right)$ in $H^1(\Omega^{\varepsilon})$, and to avoid the case $\alpha>0$ where the error bound is arbitrarily slow. Since we are now dealing with the macroscopic framework, we can get a better ``stability analysis". This argument is shown in the following theorem as a stability analysis of the macroscopic scheme in $L^{\infty}(\Omega)$. Additionally, this can be applied to Theorem \ref{thm:4.2} where the assumption \eqref{dk2} is no longer necessary by not using the Poincar\'e inequality.

\begin{theorem}
	Consider the case $\alpha = 0$. Suppose that the internal source $f$ and the regularization  $\mathcal{R}_{\gamma_{k}}$ are smooth. Under the assumptions of Theorem \ref{thm:4.2}, if we can choose
	\begin{align}
	\gamma_{k}= \frac{C}{k+1},
	\end{align}
	the iterative sequence $\left\{u_{0}^{k}\right\}_{k\in\mathbb{N}^{*}}$ is  stable in the sense that
	\begin{align*}
		\left\Vert u_{0}^{k}-u_{0}^{k-1}\right\Vert _{L^{\infty}\left(\Omega\right)}\le \frac{C}{(k+1)^{\sigma -1}}, \quad\text{with }\sigma> 1.
	\end{align*}
\end{theorem}
\begin{proof}
	In this proof, we also follow the same vein as proof of Theorem \ref{thm:4.2} where we consider two typical structures of the limit equation \eqref{eq:4.2} and investigate their corresponding difference equation by the standard variational setting. For brevity, we do not mention again those equations, but the governing difference equations. Due to the linear problem as well as the smoothness of $f$ and $\mathcal{R}_{\gamma_{k}}$, it is also trivial to prove that the functions $u_0^k$ are in $L^{\infty}(\Omega)$, cf. \cite{Agmon1959}. In this regard, we recall the difference equation
	\begin{align}\label{eq:4.6}
	\bar{a}\left(v_{0}^{k},\bar{\varphi}\right)+\left(\eta+\delta_{1}\right)\left|Y_{l}\right|\left\langle v_{0}^{k},\bar{\varphi}\right\rangle _{L^{2}\left(\Omega\right)} & =\left|Y_{l}\right|\left\langle \mathcal{R}_{\gamma_{k-1}}\left(u_{0}^{k-2}\right)-\mathcal{R}_{\gamma_{k}}\left(u_{0}^{k-2}\right),\bar{\varphi}\right\rangle _{L^{2}\left(\Omega\right)}\\ \nonumber
	& +\left|Y_{l}\right|\left\langle \bar{g}_{\gamma_{k}}\left(u_{0}^{k-2}\right)-\bar{g}_{\gamma_{k}}\left(u_{0}^{k-1}\right),\bar{\varphi}\right\rangle _{L^{2}\left(\Omega\right)},
	\end{align} 
	where the test function $\bar{\varphi}\in H_0^1(\Omega)\cap L^{\infty}(\Omega)$ is now taken into account.
	
	Now we set
	\begin{align}
	D_{k} := \frac{1}{\eta+\delta_{1}}\left\Vert \left(\eta+\delta_{1}\right)v_{0}^{k-1}+\mathcal{R}_{\gamma_{k-1}}\left(u_{0}^{k-2}\right)-\mathcal{R}_{\gamma_{k}}\left(u_{0}^{k-1}\right)\right\Vert _{L^{\infty}\left(\Omega\right)}>0,
	\end{align}
	and put
	\begin{align}
	W_{k}^{1}:=\left\{ x\in\Omega:v_0^{k}+ D_{k} < 0\right\}.
	\end{align}
	Then we assume that $\left|W_{k}^{1}\right|>0$. By choosing $\bar{\varphi} = \left(v_0^{k}+ D_{k}\right)^{-}$ in \eqref{eq:4.6} where $f^{-} := \min\left\{f,0\right\}$, we project \eqref{eq:4.6} from $\Omega$ to the set $W_{k}^{1}$. Thus, below we will accompany the notation $W_{k}^{1}$ to indicate the fact that we are working in that set, although the essential integral has to be posed in $\Omega$. In fact, \eqref{eq:4.6} now becomes
	\begin{align}\label{eq:4.7}
	&\bar{a}_{W_k^1}\left(v_{0}^{k},\bar{\varphi}\right)
	+\left(\eta+\delta_{1}\right)\left|Y_{l}\right|\left\langle v_{0}^{k},\bar{\varphi}\right\rangle _{L^{2}\left(W_k^1\right)} \\ \nonumber & =\left|Y_{l}\right|\left\langle \mathcal{R}_{\gamma_{k-1}}\left(u_{0}^{k-2}\right)-\mathcal{R}_{\gamma_{k}}\left(u_{0}^{k-2}\right),\bar{\varphi}\right\rangle _{L^{2}\left(W_k^1\right)} +\left|Y_{l}\right|\left\langle \bar{g}_{\gamma_{k}}\left(u_{0}^{k-2}\right)-\bar{g}_{\gamma_{k}}\left(u_{0}^{k-1}\right),\bar{\varphi}\right\rangle _{L^{2}\left(W_k^1\right)}.
	\end{align} 
	
	At this stage, we see that the bilinear form $\bar{a}_{W_k^1}\left(v_0^k,\bar{\varphi}\right)$ is non-negative certainly due to the presence of the gradient. Furthermore, it holds for a.e. $x\in W_k^1$ that
	\begin{align}
	v_{0}^{k}-v_{0}^{k-1}-\frac{1}{\eta+\delta_{1}}\left(\mathcal{R}_{\gamma_{k-1}}\left(u_{0}^{k-2}\right)-\mathcal{R}_{\gamma_{k}}\left(u_{0}^{k-1}\right)\right)\le v_{0}^{k}+D_{k}<0,
	\end{align}
	which implies
	\begin{align}
	\left(\eta+\delta_{1}\right)\left|Y_{l}\right|\left\langle v_{0}^{k}-v_{0}^{k-1}-\frac{\mathcal{R}_{\gamma_{k-1}}\left(u_{0}^{k-2}\right)-\mathcal{R}_{\gamma_{k}}\left(u_{0}^{k-1}\right)}{\eta+\delta_{1}},\bar{\varphi}\right\rangle _{L^{2}\left(W_{k}^{1}\right)}>0.
	\end{align}
	This essentially contradicts the equality \eqref{eq:4.7} and thus the assumption $\left|W_{k}^{1}\right|>0$ does not hold. Equivalently, it means that for a.e. $x\in\Omega$, it holds $u_{0}^{k}-u_{0}^{k-1}\ge-D_{k}$.
	
	In the same vein, we can prove that $u_{0}^{k}-u_{0}^{k-1}\le D_{k}$ by setting
	\begin{align}
	W_{k}^{2}:=\left\{ x\in\Omega:v_0^{k} - D_{k} > 0\right\},
	\end{align}
	and employing the test function $\bar{\varphi}=\left(v_0^k - D_k\right)^{+}$ where $f^{+}:= \max\left\{f,0\right\}$. For brevity, we omit the details and leave it to the reader.
	
	Henceforward, we obtain
	\begin{align*}
	 \left\Vert u_{0}^{k}-u_{0}^{k-1}\right\Vert _{L^{\infty}(\Omega)}\le D_{k} \le\frac{1}{\eta+\delta_{1}}\left\Vert \mathcal{R}_{\gamma_{k-1}}\left(u_{0}^{k-2}\right)-\mathcal{R}_{\gamma_{k}}\left(u_{0}^{k-2}\right)+\bar{g}_{\gamma_{k}}\left(u_{0}^{k-2}\right)-\right.\\
	 \left.\bar{g}_{\gamma_{k}}\left(u_{0}^{k-1}\right)\right\Vert _{L^{\infty}\left(\Omega\right)}.
	\end{align*}
	In view of the fact that $\left|\bar{g}_{\gamma_k}'\right| \le \eta + \delta_1 -\gamma_k\delta_0$ and $\left|\mathcal{R}_{\gamma_{k-1}}-\mathcal{R}_{\gamma_{k}}\right|\le C\gamma_{k-1}^{\sigma}$, it reveals
	\begin{align}
	\left\Vert u_{0}^{k}-u_{0}^{k-1}\right\Vert _{L^{\infty}\left(\Omega\right)}\le\frac{\gamma_{k-1}^{\sigma}}{\eta+\delta_{1}}+\left(1-\frac{\gamma_{k}\delta_{0}}{\eta+\delta_{1}}\right)\left\Vert u_{0}^{k-1}-u_{0}^{k-2}\right\Vert _{L^{\infty}\left(\Omega\right)}.
	\end{align}
	
	Thanks to Lemma \ref{lem:3.1}, one then deduces
	\begin{align}\label{eq:4.8}
	\left\Vert u_{0}^{k}-u_{0}^{k-1}\right\Vert _{L^{\infty}\left(\Omega\right)}\le\bar{a}_{k}+\sum_{j=2}^{k-1}\bar{a}_{j}\prod_{i=j+1}^{k}\bar{b}_{i}+\left\Vert u_{0}^{1}\right\Vert _{L^{\infty}\left(\Omega\right)}\prod_{i=2}^{k}\bar{b}_{i},
	\end{align}
	where
	\begin{align}
	\bar{a}_{k}:=C\gamma_{k-1}^{\sigma},\;\bar{b}_{k}:=1-\frac{\gamma_{k}\delta_{0}}{\eta+\delta_{1}}.
	\end{align}
	
	It suffices to take into account the second term on the right-hand side of \eqref{eq:4.8}. Observe that we can bound it from above by using the standard inequality $1+x \le \text{exp}(x)$ for $x\in \mathbb{R}$. Indeed, it is aided by the sum and integral inequality\footnote{$\int_{j}^{k}\gamma\left(x\right)dx\le\sum_{i=j+1}^{k}\gamma\left(i\right)\le\int_{j+1}^{k+1}\gamma\left(x\right)dx$
		for any non-decreasing $\gamma$.} that
	\begin{align*}
	\sum_{j=2}^{k-1}\bar{a}_{j}\prod_{i=j+1}^{k}\bar{b}_{i} & \le C\sum_{j=2}^{k-1}\frac{1}{j^{\sigma}}\exp\left(-\frac{\delta_{0}}{\eta+\delta_{1}}\sum_{i=j+1}^{k}\frac{1}{i+1}\right)\\
	& \le C\sum_{j=2}^{k-1}\frac{1}{j^{\sigma}}\exp\left(\frac{\delta_{0}}{\eta+\delta_{1}}\log\frac{j+2}{k+1}\right) \le C\frac{1}{\left(k+1\right)^{\frac{\delta_{0}}{\eta+\delta_{1}}}}\sum_{j=2}^{k-1}\frac{\left(j+2\right)^{\frac{\delta_{0}}{\eta+\delta_{1}}}}{j^{\sigma}}.
	\end{align*}
	
	We now employ the elementary inequality $\left(j+2\right)^{\sigma}\le2^{\sigma-1}\left(j^{\sigma}+2^{\sigma}\right)\le Cj^{\sigma}$
	for $j\ge2$ to arrive at
	\begin{align*}
	\sum_{j=2}^{k-1}\bar{a}_{j}\prod_{i=j+1}^{k}\bar{b}_{i} & \le\frac{C}{\left(k+1\right)^{\frac{\delta_{0}}{\eta+\delta_{1}}}}\sum_{j=2}^{k-1}\left(j+2\right)^{\frac{\delta_{0}}{\eta+\delta_{1}}-\sigma}\le\frac{C}{\left(k+1\right)^{\frac{\delta_{0}}{\eta+\delta_{1}}}}\int_{2}^{k}\left(x+2\right)^{\frac{\delta_{0}}{\eta+\delta_{1}}-\sigma}dx\\
	& \le\frac{C\left(k+2\right)^{\frac{\delta_{0}}{\eta+\delta_{1}}-\sigma+1}}{\left(k+1\right)^{\frac{\delta_{0}}{\eta+\delta_{1}}}}\le C\left(k+2\right)^{-\sigma+1},
	\end{align*}
	by virtue of $\sigma>\frac{\delta_{0}}{\eta+\delta_{1}}$ and $1<\frac{k+2}{k+1}<2$.
	
	Hence, this way we complete the proof of the theorem.
\end{proof}

Since the homogenized system is linearized and associated with constant coefficients, one can obtain the error estimate between $u_{\varepsilon}^{k}$ and $u_{0}^{k}$ in $H^1(\Omega^{\varepsilon})$ by adapting the classical two-scale asymptotic expansion
\[u_{\varepsilon}^{k}(x) = u_{0}^{k}(x)+\varepsilon u_{1}^{k}\left(x,\frac{x}{\varepsilon}\right) +\varepsilon^2u_{2}^{k}\left(x,\frac{x}{\varepsilon}\right) + \ldots\]

\begin{theorem}\label{thm:4.1}
Assume $\left(\text{A}_{1}\right)$ holds. Suppose that $\partial\Omega\in C^6$ and $f,\mathcal{R}(0)\in H^4(\Omega)$. Then there exists $C>0$ independent of $\varepsilon$ such that
	\begin{align}\label{correct}
	\left\Vert u_{\varepsilon}^{k}-u_{0}^{k}\right\Vert_{H^{1}\left(\Omega^{\varepsilon}\right)} \le C\left(\left\Vert u_{0}^{k-1}\right\Vert _{H^{4}\left(\Omega\right)}\right)\varepsilon^{1/2}. 
	\end{align}
\end{theorem}
\begin{proof}
The proof was already obtained via \cite[Corollary 2.30]{Cioranescu1999}. Here, we single out how that proof works in our particular case. Starting from $k=1$, our solution $u_{0}^{1}$ belongs to $H^6(\Omega)$ using the standard higher boundary regularity result for elliptic equations (cf. \cite[Section 6.3]{Evans2010}). Since $\mathcal{R}\left(0\right)\in H^{4}\left(\Omega\right)$, we get $\mathcal{R}_{\gamma_{k}}\left(u_{0}^{k-1}\right) \in H^4(\Omega)$ in the sense that
	\begin{align*}
	\left\Vert \mathcal{R}_{\gamma_{k}}\left(u_{0}^{k-1}\right)\right\Vert _{H^{4}\left(\Omega\right)} & \le\delta_{1}\left\Vert u_{0}^{k-1}\right\Vert _{H^{4}\left(\Omega\right)}+\left\Vert \mathcal{R}_{\gamma_{k}}\left(0\right)\right\Vert _{H^{4}\left(\Omega\right)}\\
	& \le\delta_{1}\left\Vert u_{0}^{k-1}\right\Vert _{H^{6}\left(\Omega\right)}+\left\Vert \mathcal{R}\left(0\right)\right\Vert _{H^{4}\left(\Omega\right)}+C.
	\end{align*}
	Therefore, we conclude that any $u_{0}^{k}$ belongs to $H^6(\Omega)$. The embedding $H^6(\Omega) \subset C^4(\overline{\Omega})$ guarantees the fact that the derivatives of $u_{0}^{k}$ up to the fourth order are bounded. Moreover, recall that cf. \cite{Savare1998}, the cell functions $\chi^{k}$ are in $H^1(Y_l)$. Hence, we can adapt the proof of \cite[Corollary 2.30]{Cioranescu1999} to prove that \eqref{correct} holds true.
\end{proof}

Ultimately, we deduce the following error estimate between $u_{\varepsilon}$ and $u_{0}^{k}$.

\begin{corollary}\label{cor:error2}
Under the assumptions of Theorem \ref{thm:3.2} and Theorem \ref{thm:4.1}, the following rate of convergence holds
	\begin{align}
	\left\Vert u_{\varepsilon}-u_{0}^{k}\right\Vert _{H^{1}\left(\Omega^{\varepsilon}\right)}\le C\left(\overline{\omega}^{k} + C_{k}\varepsilon^{1/2}\right),
	\end{align}
	where $C_{k}=C\left(\left\Vert u_{0}^{k-1}\right\Vert _{H^{4}\left(\Omega\right)}\right)>0$.
\end{corollary}

In Corollary \ref{cor:error2}, we see that due to the presence of $C_{k}$ it is not necessary to take $k$ very large to get a fine approximation of $u_{\varepsilon}$. Theoretically, this $C_k$ largeness might slow down the smallness of the quantity $\varepsilon^{1/2}$. From the computational standpoint, usually one should stop at $k=5$ at most because the algorithm is robust and it avoids being time-consuming.

\section{Numerical implementation}\label{Sec:numerical}

In this section, we numerically investigate the potential of the developed iterative method in approximating nonlinear elliptic problems, which are described on complex porous domains. Here, we focus on the case $\alpha = 0$ since this case not only remains nonlinear in the macroscopic problem, but also can be considered as a paradigm for treating other problems of interest. Besides, the case, $\alpha >0$, is simple and can be handled in a more straightforward manner since the reaction term converges to $0$ as $\varepsilon$ tends $0$.

 For the purpose of illustration, let us consider the problem  \eqref{Eq(P)} in a two-dimensional unit square $\Omega =(0,1)\times (0,1)$. The highly oscillatory diffusion coefficient is chosen as
\begin{align*}
\mathbf{A}(x/\varepsilon) = \dfrac{1}{2+\cos\left(\frac{2\pi x}{\varepsilon} \right)\cos(\frac{2\pi y}{\varepsilon})},
\end{align*}
and we take $\mathcal{R}(u)$  as in \eqref{eq:3.7} with $p = 2$, and the source term is $f = 1$.

In order to demonstrate the $L^2$ estimate of the error between $u_{\varepsilon}$ and $u_{\varepsilon}^k$, we first solve the  problem \eqref{eq:4.2} for $u_{\varepsilon}^k$, which involve solving \eqref{eq:cell} and \eqref{eq:homogenized} for the cell functions $ \chi_i, i=1,2,$ and for the effective diffusion coefficient, $A^0$, respectively. We consider \eqref{eq:cell} in the unit cell $Y =(0,1)\times (0,1)$ with a hole of radius $r = 0.4$ and  porosity, $|Y_l | = 1- \pi r^2$. Moreover, we take the constants $\eta = 0.4$ and $\delta_1 = 1$ for the stabilization constant $M$, cf. Theorem \ref{thm:3.1-1}. Note that the regularization $\mathcal{R}$ is given, according to  \eqref{eq:3.7}, by
\begin{align}\label{eq:regular}
\mathcal{R}_{\gamma_{k}}\left(u\right)=\max\left\{ u^{2},\frac{\delta_{0}}{\gamma_{k}}u\right\}, \quad \gamma_{k} = \frac{1}{2^{k+2}}.
\end{align}
Eventually, by plugging the cell solutions $\chi_i, i=1,2$ in \eqref{eq:homogenized}, we can compute the homogenized diffusion coefficient, as follows:
\begin{align}
A^0 = \begin{pmatrix} 0.192688 &  1.89291\times10^{-8}  \\
 1.89291\times10^{-8} & 0.192688 \end{pmatrix}.
\end{align}

\begin{table}[H]
    \centering
    \caption{The errors: $E_1$ between $u_{\varepsilon}$ and $u_0^k$, and between $ \nabla u_{\varepsilon}$ and $\nabla  u_0^k$; the relative error $E2$ between $u_{\varepsilon}$ and $u_0^k$   (with $k=2)$}
      \begin{tabular}{  +c +c +c +c +c +c  +c +c}
  \hline
    $\varepsilon$&0.5& 0.25&$ 0.166$ & 0.1 & 0.083 & 0.05&0.025 \\ \hline   
  $E_{1}\left(u_{\varepsilon,h},u_{0,h}^{k}\right)$ 	&0.0113	&0.0040 &	0.0026 &	0.0018	&	0.0017	&	0.0016	&	0.00158 \\\hline

$E_{1}\left(\nabla u_{\varepsilon,h},\nabla u_{0,h}^{k}\right)$ &0.2167 & 	0.1857& 	0.1824	&0.1761		&0.1755		&0.1746		& 0.17425	 \\\hline
 $E_{2}\left(u_{\varepsilon,h},u_{0,h}^{k}\right)$ &0.1854& 	0.0562& 0.0357	 &	0.0242	&	0.0230		& 0.0217	&	0.02128	 \\\hline
   $\max_{h}$ &  0.04 & 0.02&  0.014& 0.008 & 0.007 &  0.006& 0.0058 \\\hline
\end{tabular}
\label{table:E1}
\end{table}

The original problem \eqref{Eq(P)} for $u_{\varepsilon}$ is solved using the Newton--Raphson method and the P1 standard finite elements on a non-uniform mesh discretization with a size $h$, satisfying $h<\varepsilon$. Since the mesh is non-uniform, we denote by $\max_{h}$ the largest value of the mesh size to qualify the condition $h<\varepsilon$. As $\varepsilon$ decreases from 0.5, $\max_{h}$ also decreases and remains consistent with the requirement, $h<\varepsilon$, for multiscale simulation of homogenization problems as shown in Table \ref{table:E1}. The iterative scheme for problems \eqref{prob: P^k_{varepsilon}} and \eqref{eq:4.2} are solved until the difference between the $L^2$ estimates of successive iterations is close to zero.
 
To assess the efficiency of the linearization of \eqref{Eq(P)} in \eqref{prob: P^k_{varepsilon}}, we take into account the following error estimates:
\[
E_{1}\left(U,V\right)=\left\Vert U-V\right\Vert _{L^{2}(\Omega^{\varepsilon})},\quad E_{2}\left(U,V\right)=\dfrac{\left\Vert U-V\right\Vert _{L^{2}(\Omega^{\varepsilon})}}{\|U\|_{L^{2}(\Omega^{\varepsilon})}}.
\]

The relative error between $u_{\varepsilon,h}^k$ and $u_{0,k}^h$ are tabulated in Table \ref{Table:E2}. It indicates that the discrepancy between the solutions of the linearized microscopic problem and the linearized macroscopic problem is good within $k=4$. However, the convergence rate of the Newton's iteration for $u_{\varepsilon,h}$ is fast when compared with the convergence rate of the linearization algorithm for $u^k_{\varepsilon,h}$. One key result, as demonstrated in Tables \ref{table:E1} and \ref{Table:E2}, is that the proposed linearization algorithm conveniently promotes the passage from the nonlinear microscopic description to the  corresponding macroscopic description, without the need of performing Taylor's expansion for nonlinear terms as is customary in the classical homogenization theory. We also remark that the Newton's iteration usually needs a fine initial guess to attain convergence, while the choice is arbitrary for our linearization scheme.

In addition, in Table \ref{table:E1}, the norm of the difference between  $ \nabla u_{\varepsilon,h} $ and $\nabla u^k_{0,h}$ in $L^2$ are listed. When $\varepsilon$ is decreasing, the error becomes smaller with $k$ up to 2. It confirms that the performance of our technique on the problem is good enough for such kind of classical homogenization problems. However, a stronger $H^1$ error estimate could be observed by considering terms of correction in the estimate, which is beyond the scope of the present manuscript.

As to our numerical results, by fixing $\varepsilon =0.1$ the relative $L^2$ error between $u_{\varepsilon,h}$ and $u_{\varepsilon,h}^k$ with $k=2$ gives $1.743 \times 10^{-2}$ and we have illustrated these solutions in Figure \ref{fig:2}.  Figure \ref{fig:3} depicts our linearized macroscopic solution at $k=4$ and the microscopic solution when $\varepsilon$ varies from $0.25$ to $0.025$. We can conclude from Figures \ref{fig:2} and \ref{fig:3} that our linearization scheme performs very well in both microscopic and macroscopic contexts. Moreover, the decrease in the error when $\varepsilon$ becomes smaller, as demonstrated in Tables \ref{table:E1} and \ref{Table:E2}, shows the consistency of our method with the classical homogenization theory, i.e. the convergence of $u_{\varepsilon}$ to $u_0$, for some $k$ in the iterative scheme.

\begin{table}[H]
    \centering
    \caption{The relative errors $E_2$ between $u_{\varepsilon,h}^k$ and $u_{0,h^k}$, between $u_{\varepsilon,h}$ and $u_{\varepsilon,h}^k$, and between $\nabla u_{\varepsilon,h}$ and $\nabla u_{\varepsilon,h}^{k}$ for the fixed value of $\varepsilon =0.25$ as $k$ increases up to 4}
      \begin{tabular}{  +c +c +c +c +c }
  \hline
    $k$ &1		& 	2	&3 & 4 \\ \hline   
  $E_{2}\left(u_{\varepsilon,h}^{k},u_{0,h}^{k}\right)$ 	& 0.06591 &		0.06441 &	0.06413 &		0.06408 \\\hline
 $E_{2}\left(u_{\varepsilon,h},u_{\varepsilon,h}^{k}\right)$ & 0.13964	& 0.01764  &	0.00219 	& 0.00027 \\\hline
 $E_{2}\left(\nabla u_{\varepsilon,h},\nabla u_{\varepsilon,h}^{k}\right)$ & 0.13623 &	0.01720 &	0.00214 &	0.00027	 \\\hline
\end{tabular}
\label{Table:E2}
\end{table}

%

  \begin{figure}[H]
\centering
\begin{subfigure}[b]{0.47\linewidth}
    \includegraphics[width=\linewidth]{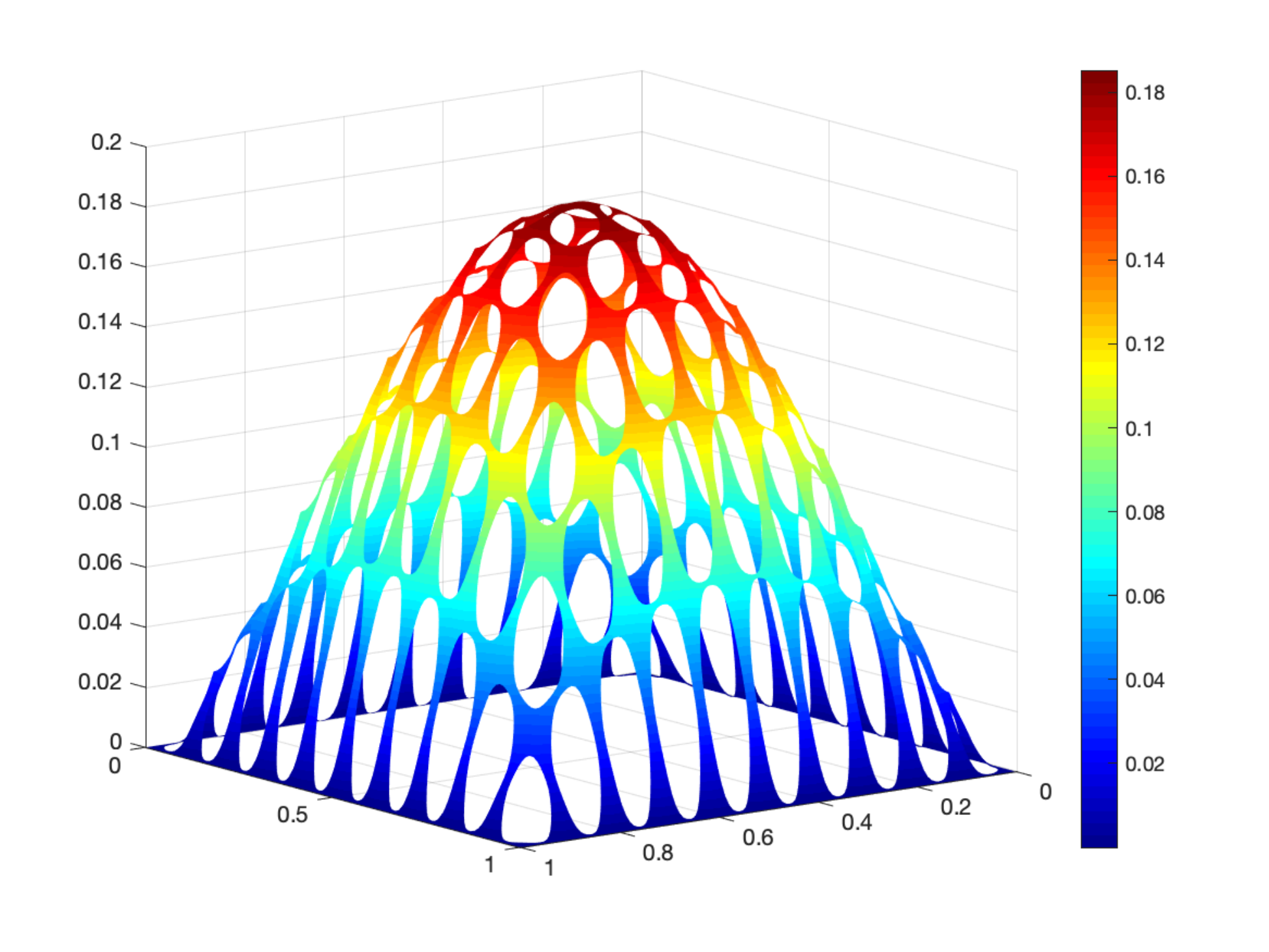}
     \caption{$u_{\varepsilon,h}$}
  \end{subfigure}
  \begin{subfigure}[b]{0.47\linewidth}
    \includegraphics[width=\linewidth]{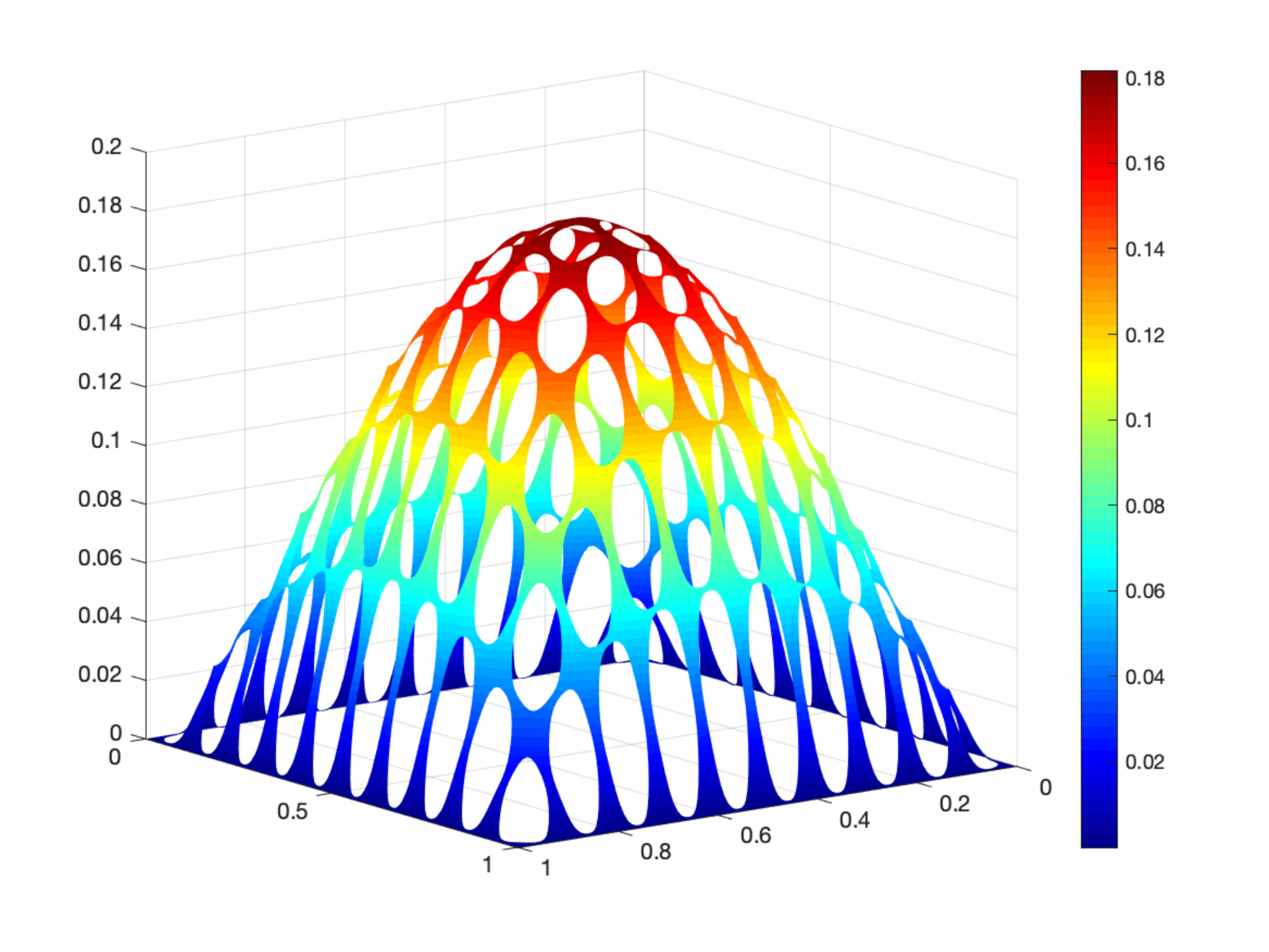}
    \caption{$u_{\varepsilon,h}^k$}
  \end{subfigure}
  \caption{Comparison of solutions between $u_{\varepsilon,h}$ and $u_{\varepsilon,h}^k$ with $\varepsilon = 0.1$ and $k=2$.}
  \label{fig:2}
\end{figure}

\begin{figure}[h]
  \centering
  \begin{subfigure}[b]{0.47\linewidth}
    \includegraphics[width=\linewidth]{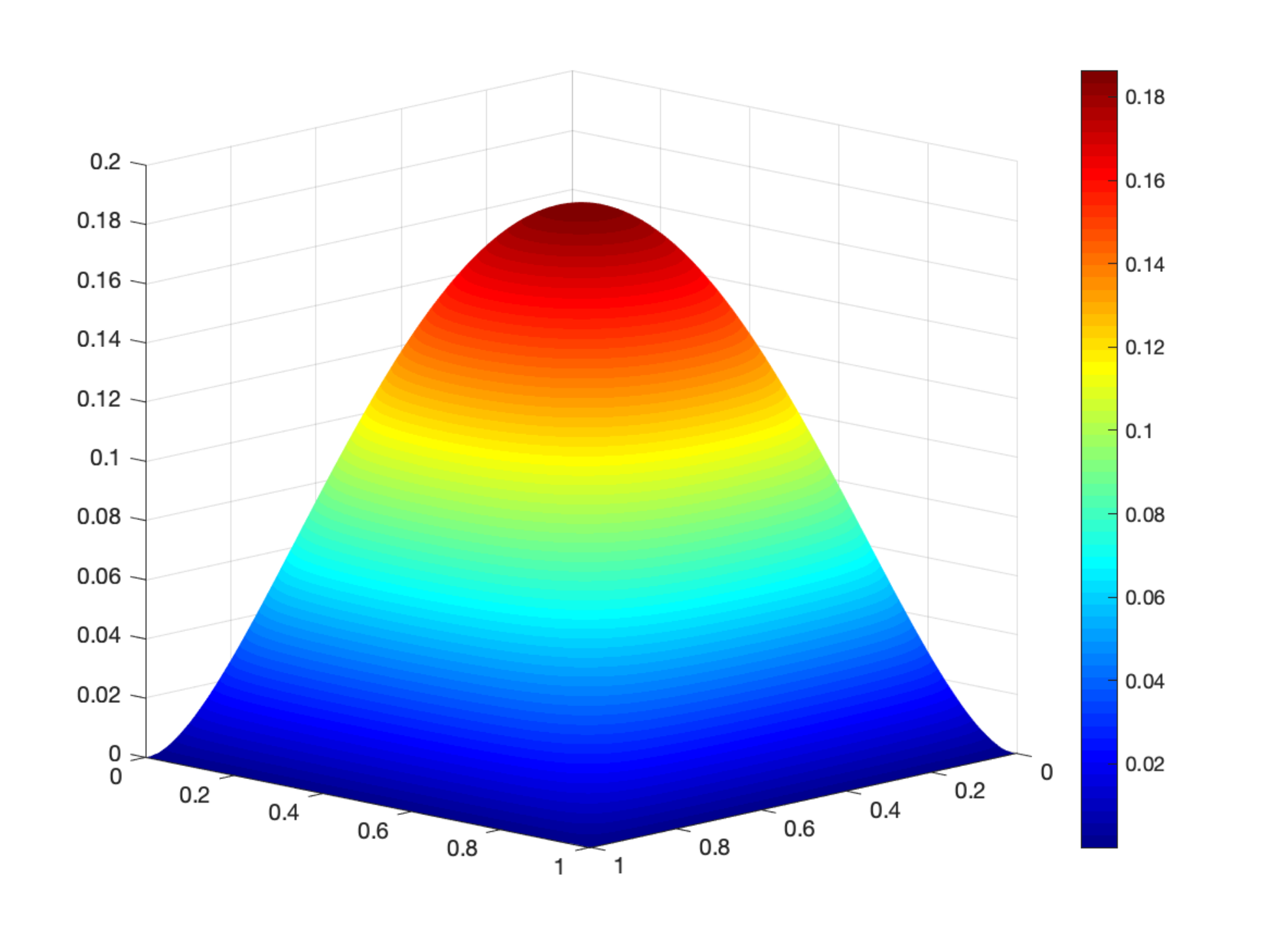}
     \caption{$u_{0,h}^k$}
  \end{subfigure}
  \begin{subfigure}[b]{0.47\linewidth}
    \includegraphics[width=\linewidth]{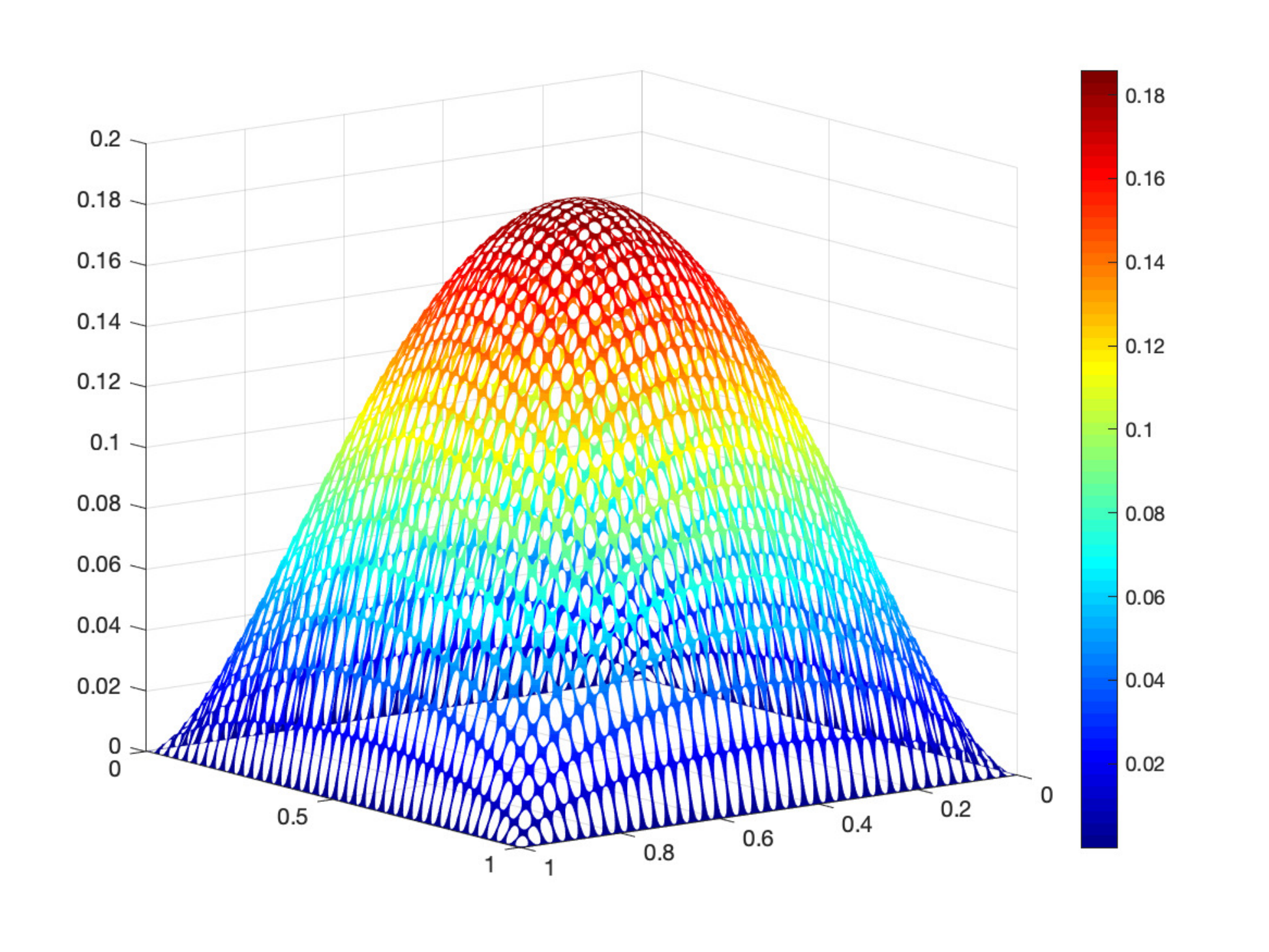}
    \caption{$\varepsilon = 0.25$}
  \end{subfigure}
  \begin{subfigure}[b]{0.47\linewidth}
    \includegraphics[width=\linewidth]{ueps01_3d}
    \caption{$\varepsilon = 0.1$}
  \end{subfigure}
  \begin{subfigure}[b]{0.47\linewidth}
    \includegraphics[width=\linewidth]{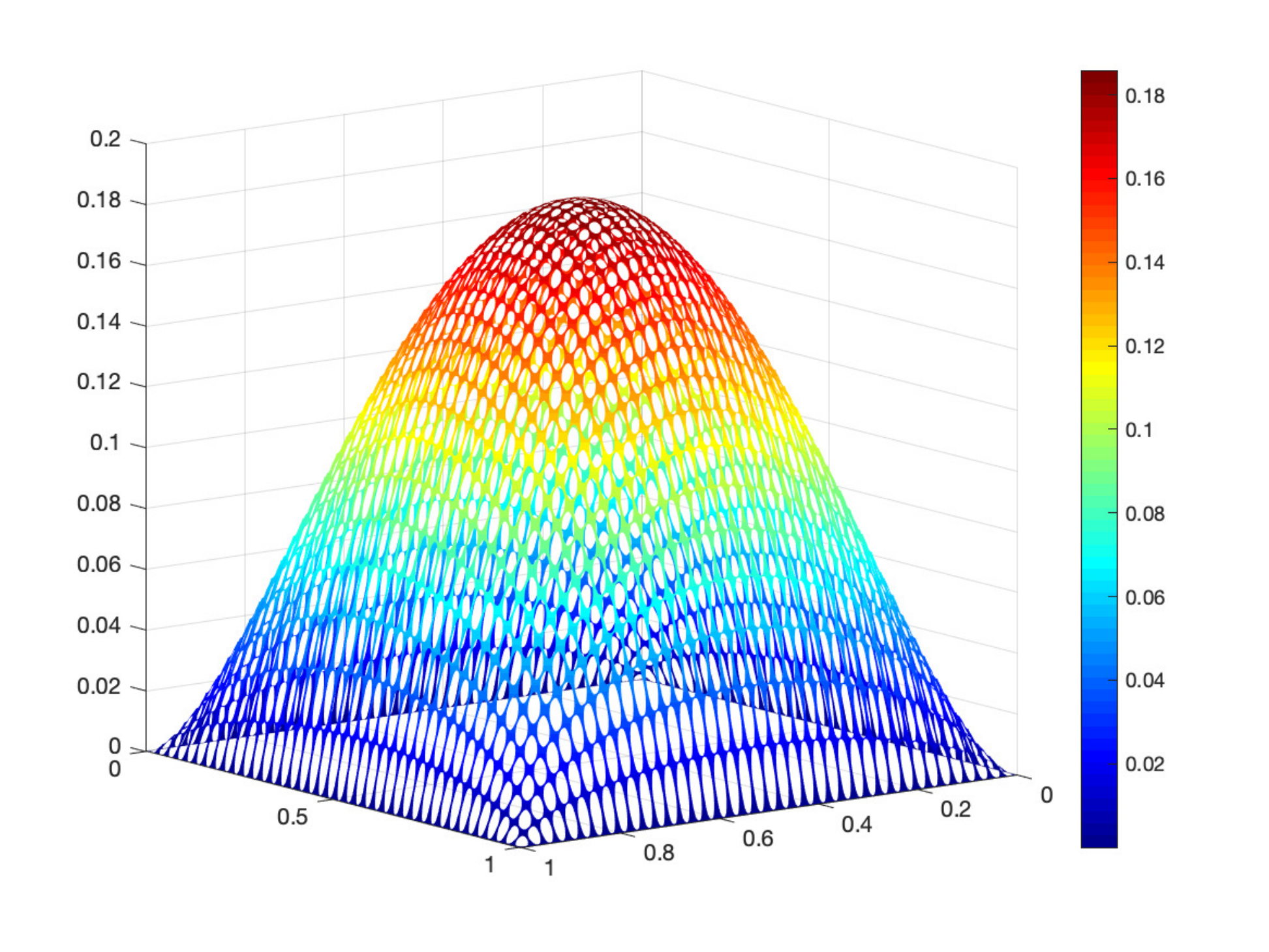}
    \caption{$\varepsilon = 0.025$}
  \end{subfigure}
  \caption{Spatial distribution of solution profiles of the macroscopic and microscopic problems. (a) the macroscopic solution $u_{0,h}^k$ with $k=2$, (b)-(d) evolution of the microscopic solution $u_{\varepsilon,h}$ as a function of $\varepsilon$.}
  \label{fig:3}
\end{figure}

\section{Concluding remarks}\label{Sec:conclusion}
In this work, we have proposed a regularization- and linearization-based scheme to construct efficient approximations of both microscopic and macroscopic problems. Although, for example, in power-law nonlinearity, a geometric regularization parameter is needed to prove the well-posedness of microscopic problem, in practice one can utilize the harmonic progression at the macro-scale to get convergence from the scheme without paying attention to the rate. Note that the harmonic choice of the regularization parameter is rather well-suited to polynomial approximations as deduced in  \eqref{hehe}. We emphasize that the arguments in this paper can be typically applied to single out a reliable approximation of the macroscopic equation if the weak solvability of the microscopic scenario is mathematically known. This approach could be helpful for engineering needs, among several types of linearization methods. Furthermore, in upcoming works we will attempt to adopt the so-called boundary layers correctors (see \cite{Versieux2006}) to our context to improve the error estimates in the numerical perspective.

\appendix
\section{Auxiliary proofs}\label{app:1}

\subsection*{Proof of Lemma \ref{lem:3.1}}
The proof is essentially done by induction. Indeed, it trivially holds for $ k=3 $. For any $ k=n $, one now suppose that 
\begin{align}\label{k=n case}
p_{n}+q_{n}\le a_{n}+\sum_{j=2}^{n-1}a_{j}\prod_{i=j+1}^{n}b_{i}+q_{1}\prod_{i=2}^{n}b_{i}.
\end{align}

Our aim is to prove that it still holds true for the case $k=n+1$, i.e.
\begin{align*}
p_{n+1} + q_{n+1}\le a_{n+1}+\sum_{j=2}^{n}a_{j}\prod_{i=j+1}^{n+1}b_{i}+q_{1}\prod_{i=2}^{n+1}b_{i}.
\end{align*}

Using \eqref{k=n case}, we derive that
\begin{align*}
p_{n+1}+q_{n+1}
&\le  a_{n+1} + b_{n+1}q_{n} \\
&\le a_{n+1} +  b_{n+1} \left( a_{n}+\sum_{j=2}^{n-1}a_{j}\prod_{i=j+1}^{n}b_{i}+q_{1}\prod_{i=2}^{n}b_{i}
\right) \le a_{n+1} + \sum_{j=2}^{n}a_{j}\prod_{i=j+1}^{n+1}b_{i}+q_{1}\prod_{i=2}^{n+1}b_{i},
\end{align*}
which completes the proof of the lemma.

\section*{Acknowledgements}
V.A.K. thanks Prof. Adrian Muntean for being his supervisor since March, 2015 and for giving him invaluable advice. V.A.K thanks Prof. Iuliu Sorin Pop (Hasselt, Belgium) for recent supports in his research career and acknowledges the hospitality of the Hasselt University during the time he is hosted as a postdoctoral fellow. N.N.N. acknowledges the support of the project INdAM Doctoral Programme in Mathematics and/or Applications Cofunded by Marie Sklodowska-Curie Actions, acronym: INdAM-DP-COFUND-2015, grant
number: 713485.


\bibliographystyle{spbasic} 
\bibliography{mybibfile}

\end{document}